\documentclass[english,reqno]{amsart}
\usepackage[shortlabels]{enumitem}
\usepackage{amsmath, appendix, ulem}
\usepackage[nobysame]{amsrefs}
\usepackage{amssymb, color}
\usepackage[margin=1in]{geometry}
\usepackage{mathrsfs}
\usepackage{graphicx}
\usepackage{subfig}
\usepackage{float}
\usepackage{epsf}
\usepackage{hyperref}
\usepackage{titletoc}
\usepackage{setspace}
\graphicspath{{../Figures/}}

\numberwithin{equation}{section}

\newtheorem{lem}{Lemma}[section]
\newtheorem{thm}{Theorem}[section]

\newtheorem{cor}[thm]{Corollary}
\theoremstyle{remark}
\newtheorem{rmk}{Remark}[section]
\newtheorem{definition}[thm]{Definition}

\renewcommand{\hat}{\widehat}
\renewcommand{\bar}{\overline}

\newtheorem{assum}{Assumption}

\newcommand{\nn}{\nonumber}
\newcommand{\R}{{\mathbb R}}

\newcommand{\dt}{ \, {\rm d} t}
\newcommand{\dx}{ \, {\rm d} x}
\newcommand{\dy}{ \, {\rm d} y}

\newcommand{\ds}{\, {\rm d} s}

\newcommand{\mc}[1]{\mathcal{#1}}

\newcommand{\BB}{\mathbb{B}}

\newcommand{\EE}{\mathbb{E}}
\newcommand{\RR}{\mathbb{R}}

\newcommand{\ud}{\,\mathrm{d}}

\newcommand{\la}{\langle}
\newcommand{\ra}{\rangle}

\newcommand{\norm}[1]{\left\lVert#1 \, \right\rVert}

\DeclareMathOperator*{\esssup}{ess\,sup}

\author{Hui Huang}
\address{Department of Mathematics, Technische Universit\"{a}t M\"{u}nchen, Boltzmannstr. 3, 85748 Garching, Germany}
\email{hui.huang@tum.de}
\author{Jinniao Qiu}
\address{Department of Mathematics and Statistics, University of Calgary, 2500 University Drive NW
	Calgary, AB, Canada,
	T2N 1N4}
\email{jinniao.qiu@ucalgary.ca}
\thanks{The research of J. Qiu is partially supported by the National Science and Engineering Research Council of Canada and the start-up funds from the University of Calgary. H. Huang is partially funded by  the DFG research project "Identification of Energies from Observations of Evolutions" (FO767/7-1).}
\begin{document}
\title[Stochastic Keller-Segel equation]{The microscopic derivation and well-posedness of the stochastic Keller-Segel equation}
\maketitle
\begin{abstract}
In this paper, we propose and study a stochastic aggregation-diffusion equation of the Keller-Segel (KS) type for modeling the chemotaxis in dimensions $d=2,3$.  Unlike the classical deterministic KS system, which only allows for idiosyncratic noises, the stochastic KS equation is derived from an  interacting particle system subject to both idiosyncratic and common noises. Both the unique existence of solutions to the stochastic KS equation and the mean-field limit result are addressed.
\end{abstract}

\section{Introduction}
Many bacteria, such as Escherichia coli, Rhodobacter sphaeroides and Bacillus subtilus are able to direct their movements according to the surrounding environment by a biased random walk.  For example, bacteria try to swim toward the highest concentration of nutrition  or to flee from poisons.  In biology, this phenomenon is called \textit{chemotaxis}, which describes the directed movement of cells and organisms in response to chemical gradients.  Chemotaxis is also observed in other biological fields, for instance the movement of sperm towards the egg during fertilization, the migration of neurons or lymphocytes, and inflammatory processes.

Mathematically, one of the most classical models for studying chemotaxis is the Keller-Segel (KS) equation that was  originally proposed in \cite{keller1970initiation} to characterize the aggregation of the slime mold amoebae. The classical parabolic-elliptic type KS equation is of the following form:
\begin{align}\label{KS}
\begin{cases}
\partial_t \rho_t=\triangle \rho_t-\chi\nabla\cdot(\rho_t\nabla c_t),\quad x\in\R^d, t>0\,,\\
-\triangle c_t=\rho_t\,,\\
\rho_0 \text{ is given}\,,
\end{cases}
\end{align}
where $\rho_t(x)$ denotes the bacteria density, and $c_t(x)$ represents the chemical substance concentration.  The constant $\chi>0$ denotes the chemo-sensitivity or response of the bacteria to the chemical substance.  From a mathematical point of view this equation displays many interesting effects and it has become a topic of intense mathematical research. An important feature of this equation is the competition between the diffusion $\triangle \rho_t$ and the nonlocal aggregation $-\chi\nabla\cdot(\rho_t\nabla c_t)$. Depending on the choice of the initial mass $m_0:=\int_{\RR^d}\rho_0(x)\dx$ and the chemo-sensitivity $\chi$, the solutions to the KS equation may exist globally or blow-up in finite time.  In particular, for sufficiently smooth initial conditions, the existence of solutions was verified by J{\"a}ger and Luckhaus \cite{jager1992explosions}:  if $m_0\chi$ is large then solutions are local in time,  and they are global in time if $m_0\chi$ is small.    For the 2-dimensional case, Dolbeault and Perthame \cite{dolbeault2004optimal}  completed the  result of \cite{jager1992explosions} by providing  an exact value for the critical mass: classical solutions to \eqref{KS} blow-up in finite time when $m_0\chi>8\pi$, and there exists a global in time solution of \eqref{KS} when $m_0\chi>8\pi$.  For the case with $m_0\chi=8\pi$,  Blanchet, Carrillo and Masmoudi \cite{blanchet2008infinite} showed that global solutions blow-up in infinite time converging
towards a delta dirac distribution at the center of mass.
There is an extensive literature on KS  systems and their variations, which is out of the scope of this paper. A comprehensive survey on known results related to the KS model from 1970 to 2000 can be found in  \cite{horstmann20031970}. We also refer to \cite{hillen2009user,perthame2006transport,biler2018mathematical} among many others for more recent developments.  

It is also well known that the KS equation \eqref{KS} can be derived from a system of interacting particles $\{(X_t^i)_{t\geq 0}\}_{i= 1}^N$ satisfying the following form of stochastic differential equations (SDEs): 
\begin{equation}\label{KSparticle}
{\rm d} X_t^i=\frac{\chi}{N-1}\sum_{j\neq i}^{N}F(X_t^i-X_t^j)\dt+\sqrt{2}{\rm d}B_t^i,\quad i=1,\cdots,N,\quad t>0\,,
\end{equation}
where the process $(X_t^i)_{t\geq 0}$ denotes the trajectory of the $i$-th particle, the function $F$ models the pairwise interaction between particles and $\{(B_t^i)_{t\geq 0}\}_{i= 1}^N$ are $N$ independent Wiener processes. The rigorous derivation of the KS equation, for example \eqref{KS}, from the microscopic particle system, e.g. \eqref{KSparticle}, through the propagation of chaos as $N\rightarrow \infty$ may be found in  \cite{HH1,HH2,fournier2015stochastic,havskovec2011convergence,huang2019learning,fetecau2019propagation,bresch2019mean}. For a review of the topic of the propagation of chaos and the mean-field limit, we refer the readers to \cite{jabin2017mean,carrillo2014derivation} and the references therein.  An asymptotic method, inspired by  Hilbert's sixth problem \cite{hilbert1902mathematical}, can also be applied to derive models at the macro-scale  (PDEs) from the underlying description at the micro-scale (particle systems); see \cite{bellomo2016multiscale,burini2019multiscale} for instance. 

However,  for the classical deterministic KS equation \eqref{KS}, the associated particle system \eqref{KSparticle} is only  subject to the idiosyncratic noises that are independent from one particle to another, and the effect of the idiosyncratic noises averages out, leading to the deterministic nature of the equation \eqref{KS}. In addition to such idiosyncratic noises,  this paper studies the particle systems allowing for common/environmental noises, and the limiting density function satisfies a stochastic partial differential equation of KS type which is new to the best of our knowledge.  Common environmental noises (such as temperature, light and sound) are intrinsic to a more realistic setting such as culturing bacteria.


Let $(\Omega, {\mc F}, ({\mc F}_t)_{t\geq 0},\mathbb P)$ be a complete filtered probability space where the $d'$-dimensional Wiener processes $\{(B_t^i)_{t\geq 0}\}_{i=1}^N$ are independent of each other as well as of a $d'$-dimensional Wiener process $(W_t)_{t\geq 0}$\footnote{The dimension of Wiener process $W$ may be different from $d'$; we assume the same dimensionality for notational simplicity.}.  The initial data $\zeta^{i}$, $i=1,2,\dots,N$ are independently and identically distributed (i.i.d.) with a common density function $\rho_0$ and are independent of $\{(B_t^i)_{t\geq 0}\}_{i=1}^{N}$ and $(W_t)_{t\geq 0}$. Denote by $(\mc F^W_t)_{t\geq 0}$ the augmented filtration generated by $(W_t)_{t\geq 0}$.

As the  mean-field limit from the  interacting particle system that allows for both idiosyncratic and common noises, the  stochastic aggregation-diffusion equation of Keller-Segel (KS) type, also called stochastic KS equation, is of the following form:
\begin{align}\label{SPDE}
\begin{cases}
\ud \rho_t
=\frac{1}{2}\sum_{i,j=1}^dD_{ij}\left(\rho_t\sum_{k=1}^{d'}(\nu^{ik}_t\nu^{jk}_t+ \sigma^{ik}_t\sigma^{jk}_t)\right)\dt-\chi\nabla\cdot(\nabla c_t\rho_t)\dt-\sum_{i=1}^{d}D_i\left( \rho_t \sum_{k=1}^{d'}\sigma^{ik}_t\ud W_t^k\right), \\
-\triangle c_t+c_t=\rho_t,\\
\rho_0 \text{ is given}\,,
\end{cases}
\end{align}
where $D_{ij}:=\frac{\partial^2}{\partial x_i\partial x_j}$, $D_{i}:=\frac{\partial}{\partial x_i}$, and the leading coefficients $\nu$ and $\sigma$ are deterministic functions from $[0,T]\times\RR^d$ to $\RR^{d\times d'}$.
%
%
One may solve the second equation for the chemical concentration:
\begin{align}
c_t=(I-\triangle)^{-1}\rho_t=\mc{G}\ast \rho_t(x), \label{c-bessel}
\end{align}
with $\mc{G}$ being the Bessel potential, and it follows that $\nabla c_t=\nabla \mc{G} \ast \rho_t$ where $\nabla \mc{G}$ is called the interaction force. The underlying regularized interacting particle system has the form:
\begin{align}\label{particle}
\begin{cases}
\ud X_t^{i,\varepsilon} =
\frac{\chi}{N-1}\sum\limits_{j\neq i}^N\nabla \mc{G}_{\varepsilon}(X_t^{i,\varepsilon}-X_t^{j,\varepsilon})\dt+\nu_t(X_t^{i,\varepsilon}) \ud B_t^i+\sigma_t(X_t^{i,\varepsilon})\ud W_t,
\quad i=1\cdots,N,\quad  t>0\,,\\
X_0^{i,\varepsilon} =\zeta^i,
\end{cases}
\end{align}
where 
$$\mc{G}_{\varepsilon}(x)=\psi_\varepsilon\ast \mc{G}(x)=\int_{\RR^d} \mc{G}(y)\psi_\varepsilon (x-y )\,dy, \quad x\in\RR^d,\,\,\varepsilon >0,$$
is the regularized Bessel potential with the mollifier function $\psi_\varepsilon(x):=\frac{1}{\varepsilon^d}\psi(\frac{x}{\varepsilon})$ satisfying
\begin{equation}
0\leq \psi \in C_c^\infty(\R^d),\quad \mbox{supp }\psi\subseteq B(0,1),\quad \int_{B(0,1)}\psi(x)\dx=1\,.
\end{equation}
 We mention here relevant work \cite{cattiaux20162,fournier2015stochastic} for the existence of solutions to  the non-mollified stochastic particle system \eqref{KSparticle}.  Especially in \cite[Proposition 4]{fournier2015stochastic}, they proved that for any $N \geq2$ and $T>0$, if $\{(X_t^i)_{t\geq 0}\}_{i= 1}^N$ is
the solution to \eqref{KSparticle}, then 
$$\mathbb{P}\left(\exists s \in[0, T], \exists 1 \leq i<j \leq N: X_{s}^{i}=X_{s}^{j}\right)>0$$
i.e. the singularity of the drift term is visited and the particle system is not clearly well-defined. Therefore in order to obtain a global strong solution to the interacting particle system, we regularize the singular force term $\nabla \mc{G}$.

In contrast with the classical KS models \eqref{KS} and \eqref{KSparticle}, which only allow for the idiosyncratic noise $(B^i_t)_{t\geq 0}$ that is independent from one particle to  another, the stochastic systems \eqref{SPDE} and \eqref{particle} are additionally subject to  common noise $(W_t)_{t\geq 0}$, accounting for the common environment where the particles evolve. This common noise leads to the stochastic integrals in stochastic KS equation \eqref{SPDE}, whose (continuous) martingale property and unboundedness result in the inapplicability of classical analysis for deterministic KS equations. In addition, the diffusion coefficients $\sigma$ and $\nu$ are time-state dependent; along the same lines, a general model may allow for diffusion incorporating  L$\acute{\text{e}}$vy type noises and/or dependence on the density (for instance, see \cite{burini2019multiscale,escudero2006fractional,huang2016well} for discussions on deterministic KS models with flux limited or fractional diffusion), although we will not seek such a generality herein.

In this paper, we first prove the existence and uniqueness results for both \textit{weak and strong} solutions to SPDE \eqref{SPDE}. Basically, over a given finite time interval $[0,T]$ when the $L^4$-norm of $\rho_0$ is sufficiently small, the weak solution exists uniquely and its regularity may be increased for regular initial value $\rho_0$ (see Theorems \ref{prop-weak-soltn} and \ref{thm-wellposedness}). Then, based on a duality analysis of forward and backward SPDE, we prove that the following stochastic differential equations (SDEs) of McKean-Vlasov type:
\begin{align}\label{SDE}
\begin{cases}
\ud Y^i_t= \chi\nabla \mc{G}\ast\rho^i_t(Y^i_t)\dt+\nu(Y^i_t) \ud B^i_t+\sigma(Y^i_t)\ud W_t,\quad i=1,\cdots,N, \quad t>0\,,\\
\rho^i_t\mbox{ is the conditional density of  }Y^i_t \mbox{ given } \mc{F}_t^W\,,\\
Y^i_0=\zeta^i,
\end{cases}
\end{align}
has a unique solution with the conditional density $\rho^i_t$ of $Y_t^i$ given the common noise $W_t$ \textit{existing and satisfying} SPDE \eqref{SPDE}; see Theorem \ref{thmSDE}. 
Here by the conditional density $\rho^i_t$ of $Y^i_t$ given $\mc{F}_t^W$, we mean that
$$\EE[Y^i_t\in\dx|\mc{F}_t^W]=\rho^i_t(x)\dx,$$ 
i.e., for any  $\varphi\in C_b(\RR^d)$, it holds that
\begin{align*}
\EE[\varphi(Y^i_t)|\mc{F}_t^W]=\int_{\RR^d}\varphi(x)\rho^i_t(x)\dx\,.
\end{align*}
 Finally, we prove that the solution $\{(X_t^{i,\varepsilon})_{t\geq 0}\}_{i=1}^N$ of the particle system \eqref{particle} well approximates that of \eqref{SDE}, which indicates the mean-field limit result, i.e.,  the empirical measure 
		\begin{equation*}
	\rho_t^{\varepsilon,N}:=\frac{1}{N}\sum_{i=1}^{N}\delta_{X_t^{i,\varepsilon}}\,,
	\end{equation*}
associated with the particle system \eqref{particle} converges weakly to  the unique solution $\rho$ to SPDE \eqref{SPDE} as $N\rightarrow \infty$ and $\varepsilon \rightarrow 0^+$; see Theorem \ref{thmmean} and Corollary \ref{cormean}.

In view of SPDE \eqref{SPDE} and the particle system \eqref{particle}, one may see that when the particle number $N$ tends to infinity,  the effect of the idiosyncratic noises averages out while the effect of common noises does not, leading to the stochastic nature of the limit distribution characterized by SPDE \eqref{SPDE}. We refer to \cite{Bensoussan2013Mean,carmona2016mean,carmona2018probabilistic,coghi2016propagation}  for different models with common noise in the literature. In particular, in a closely related work \cite{coghi2016propagation}, the authors study the propagation of chaos for an interacting particle system subject to a common environmental noise but with a uniformly Lipschitz continuous potential, and in  \cite{choi2019cucker}, the stochastic mean-field limit of the Cucker-Smale flocking particle system is obtained for a special class of noises. In contrast to the existing literature concerning common noise, the main difficulties in dealing with the proposed stochastic KS models are from the Bessel potential $\mc{G}$ which entails the singularity of the drift of SDE \eqref{SDE} and the KS type nonlinear and nonlocal properties of SPDE \eqref{SPDE}; in particular, the KS type nonlinear term $-\chi \nabla\cdot((\nabla\mc{G}\ast \rho_t) \rho_t)$ prevents us from adopting the existing methods in the SPDE literature. Accordingly, the existence and uniqueness of solution to SPDE \eqref{SPDE} is established within sufficiently regular spaces under a divergence-free assumption on coefficient $\sigma$, and we prove that the conditional density exists and satisfies  equation \eqref{SPDE} with a new method based on duality analysis. In addition, for the mean-field limit result,  we also introduce regularization with a mollifier function in the particle system \eqref{particle}. In this paper, the approaches mix and develop the existing probability theory and stochastic analysis, (S)PDE theory,  and the duality analysis in nonlinear filtering theory.  Given the outstanding interests shown in the mathematical analysis of biological phenomena, we hope this article will set the stage for further studies on stochastic aggregation-diffusion type equations, opening new perspectives and motivating applied mathematicians to expand the research on this class of models to novel applications.

The rest of the paper is organized as follows. In Section 2, we set some notations, present some auxiliary results and give the standing assumptions on the diffusion coefficients. Section 3 is then devoted to the proof of the existence and uniqueness of the \textit{weak and strong} solution to stochastic KS equation \eqref{SPDE} in certain regular spaces. On the basis of the well-posedness of SPDE \eqref{SPDE}, we prove the existence and uniqueness of the strong solution to SDE \eqref{SDE} in Section 4. Finally, the mean-field limit result is addressed in Section 5.


\section{Preliminaries}

\subsection{Notations}
The set of all the integers is denoted by $\mathbb{Z}$, with $\mathbb{Z}^+$ the subset of the strictly positive elements. Denote by $|\cdot|$ (respectively,
$\langle \cdot,\cdot \rangle$ or $\cdot$) the usual norm (respectively, scalar product) in
finite-dimensional Hilbert space such as $\mathbb R,\mathbb R^k,\mathbb R^{k\times l}$, $k,l\in \mathbb{Z}^+$.  We use $\|f\|_p$ for the $L^p$ $(1\leq p\leq \infty)$ norm of a function $f$. 

Define the set of multi-indices
$$\mathcal {A}:=\{\alpha=(\alpha_1,\cdots,\alpha_d): \alpha_1, \cdots, \alpha_d \textrm{
are nonnegative integers}\}.$$ For any $\alpha\in \mathcal{A}$ and
$x=(x_1,\cdots,x_d)\in \RR^d,$ denote
$$ |\alpha|=\sum_{i=1}^d
\alpha_i,\ x^{\alpha}:=x_1^{\alpha_1}x_2^{\alpha_2}\cdots x_d^{\alpha_d},\ D^{\alpha}:=\frac{\partial^{|\alpha|}}
{\partial x_1^{\alpha_1}\partial x_2^{\alpha_2}\cdots\partial x_d^{\alpha_d}}. $$

For each Banach space $(\mathcal {X},\|\cdot\|_{\mathcal {X}})$,  real $q\in[1,\infty]$, and $0\leq t <\tau\leq T$, we denote by $S_{\mathcal F}^q([t,\tau];\mathcal{X})$ the set of $\mathcal{X}$-valued, $\mathcal F_t$-adapted and continuous processes $\{X_s\}_{s\in[t,\tau]}$ such that
\begin{equation*}
\|X\|_{S_{\mathcal F}^q([t,\tau];\mathcal{X})}:=
\left\{\begin{array}{l}
\left(\EE\Big[  \sup_{s\in[t,\tau]} \|X_s\|_{\mathcal{X}}^q  \Big]\right)^{1/q},\quad q\in[1,\infty);\\
\esssup_{\omega\in\Omega}  \sup_{s\in[t,\tau]} \|X_s\|_{\mathcal{X}},\quad q=\infty.
\end{array}\right.
\end{equation*}
$L_{\mathcal F}^q(t,\tau;\mathcal{X})$ denotes the set of (equivalent classes of) $\mathcal{X}$-valued predictable processes
$\{X_s\}_{s\in[t,\tau]}$ such that
\begin{equation*}
\|X\|_{L^q_{\mathcal F}(t,\tau;\mathcal{X})}:=
\left\{\begin{array}{l}
\left(\EE\Big[  \int_t^\tau \|X_s\|_{\mathcal{X}}^q\,ds  \Big]\right)^{1/q},\quad q\in[1,\infty);\\
\esssup_{(\omega,s)\in\Omega\times{[t,\tau]}} \|X_s\|_{\mathcal{X}},\quad q=\infty.
\end{array}\right.
\end{equation*}
Both $\left(S_{\mathcal F}^q([t,\tau];\mathcal{X}),\|\cdot\|_{S^q([t,\tau]\mathcal{X})}\right)$ and $\left(L_{\mathcal F}^q(t,\tau;\mathcal{X}),\|\cdot\|_{L_{\mathcal F}^q(t,\tau;\mathcal{X})}\right)$ are Banach spaces, and they are well defined with the filtration $(\mc F_t)_{t\geq 0}$ replaced by $(\mc F_t^W)_{t\geq 0}$.

\subsection{Auxiliary results and assumptions}
We first recall some properties of the Bessel potential introduced in \eqref{c-bessel}. For $p\in[1,\infty]$, denote by $L^p=L^p(\RR^d)$ the usual Lebesgue integrable spaces with norm $\|\cdot\|_p$.  Then for $p\in (1,\infty)$ and $m\in \RR$, we may define the space of Bessel potentials (or the Sobolev space with fractional derivatives) \cite[p. 37]{Trieb} as 
$$\mc H_p^m(\RR^d)=\left\{f\big|~\|f\|_{\mc H_p^m}:=\|\mathscr{F}^{-1}[(1+|\omega|^2)^{\frac{m}{2}}\mathscr{F}(f)]\|_p<\infty\right\}\,,$$
where $\mathscr{F}$ is the Fourier transformation. Namely, $\mc H_p^m(\RR^d)$ (simply written as $\mc H_p^m$) is defined as space of functions $f$ such that $(1-\triangle)^{\frac{m}{2}}f\in L^p(\RR^d)$. 
In \eqref{c-bessel}, if $\rho_t\in L^p$ with $1<p<\infty$, then $c_t \in \mc{H}_p^2$. In addition, it holds that 
\begin{align*}
\norm{c_t}_{\mc{H}_p^2}=\norm{\mathscr{F}^{-1}\left[(1+|\omega|^2)\mathscr{F}[c_t]\right]}_{p}=\norm{\rho_t}_{p}\,.
\end{align*}
  Due  to the equivalence between the Bessel potential space $\mc{H}_p^k(\RR^d)$  and the Sobolev space $W^{k,p}(\RR^d)$ $(k\in \mathbb{N})$, we have
\begin{equation}\label{cestimate}
\norm{\mc{G}\ast \rho_t}_{W^{2,p}}=\norm{c_t}_{W^{2,p}}\leq C\norm{\rho_t}_{p}.
\end{equation}
 Here  the Sobolev space $W^{{k,p}}(\RR^d )$  is defined as
$$W^{k,p}(\RR^d )=\left\{u\in L^{p}(\RR^d )\big|~D^{\alpha }u\in L^{p}(\RR^d ),\,\,\forall~ |\alpha |\leq k\right\}$$
and
\begin{align*}
\|u\|_{W^{k, p}}:=\left\{
\begin{array}{ll}\left(\sum_{|\alpha| \leqslant k}\left\|D^{\alpha} u\right\|_{L^{p}(\RR^d)}^{p}\right)^{\frac{1}{p}} & 1 \leqslant p<\infty,
 \\ \max _{|\alpha| \leqslant k}\left\|D^{\alpha} u\right\|_{L^{\infty}(\RR^d)} & p=\infty.
 \end{array}\right.
\end{align*}

On the other side, notice that
\begin{equation*}
(I-\triangle)^{-1}=(-\triangle)^{-1}-(-\triangle)^{-1}(I-\triangle)^{-1}\,.
\end{equation*}
Thus,  we may split the Bessel potential into the Newtonian potential $\Phi$ and a function $\Psi$ such that $\mathscr{F}(\Psi)(\omega)=-\frac{1}{\omega^2(1+\omega^2)}$, which implies that
 $\Psi\in L^\infty(\RR^d)$ $(d=3)$ or $\nabla\Psi\in L^\infty(\RR^d)$ $(d=2)$. 
Namely, one has
\begin{equation}\label{Gsplit}
\mc{G}(x)=\Phi(x)+\Psi(x)\,,
\end{equation}
where $$
\Phi(x)=\left\{\begin{array}{ll}{\frac{C_{d}}{|x|^{d-2}},} & {\text { if } d \geq 3} \\ {-\frac{1}{2 \pi} \ln |x|,} & {\text { if } d=2}\end{array}\right.$$
 is the Newtonian potential.  It then follows that for any $\alpha\in\mathcal A$ with $|\alpha|\geq 1$, there holds
\begin{align}
\norm{D^\alpha(\nabla\mc{G}_\varepsilon)}_\infty
&\leq C_{\alpha}\varepsilon^{1-d-|\alpha|}+
\begin{cases}
C_{\alpha,\norm{\varPsi}_\infty}\varepsilon^{-1-|\alpha|},\quad &\mbox{when }d=3\,;\\
C_{\alpha,\norm{\nabla\varPsi}_\infty}\varepsilon^{-|\alpha|},\quad &\mbox{when }d=2
\end{cases}
\nonumber \\
&\leq C_\alpha\varepsilon^{1-d-|\alpha|}\,.\label{espotential}
\end{align}
Here, we have used the estimate $\norm{D^\alpha(\nabla\Phi_\varepsilon)}_\infty\leq C_{\alpha}\varepsilon^{1-d-|\alpha|}$ from \cite[Lemma 2.1]{HH1}.

Following are the standing assumptions on the coefficients $\nu$ and $\sigma$.
\begin{assum}\label{assum1} Given $T>0$  any arbitrary time horizon  and $d=2,3$, the measurable diffusion coefficients 
$\sigma,\,\nu: [0,T]\times \RR^d \longrightarrow \RR^{d\times d'}$
satisfy
	\begin{enumerate}[(i)]
		\item There exists a positive constant $\lambda$ such that
		$$
		\sum_{i,j=1}^d \sum_{k=1}^{d'}\nu^{ik}_t(x)\nu^{jk}_t(x)\xi^i\xi^j\geq \lambda |\xi|^2
		$$
		holds for all $x,\xi\in\mathbb R^d$ and all $t\geq 0$;\\
		\item There exist $m\in\mathbb Z^+$ and real $\Lambda >0$ such that for all $t\in[0,T]$ there holds
		$$
		\nu^{ik}_t(\cdot),\,\sigma^{ik}_t(\cdot)\in C^m , \,\,\text{for } i=1,\dots,d, \,k=1,\dots,d',\quad \text{and}\quad \sum_{i=1}^d\sum_{k=1}^{d'}
		\|\sigma^{ik}_t(\cdot)\|_{C^m}+\|\nu^{ik}_t(\cdot)\|_{C^m} \leq \Lambda,
		$$
		 where the $C^m$ norm is defined as
		$\|f\|_{C^m}=\sum_{|\alpha|\leq m}\|D^\alpha f\|_\infty$.
		\\
		\item For all $(t,x)\in[0,T]\times \mathbb R^d$ and $k=1,2,\dots,d'$,
		$$
		\sum_{i=1}^dD_i\sigma^{ik}_t(x)=0.
		$$
%
		
	\end{enumerate}
\end{assum}
\begin{rmk}\label{rmk-ass}
	The assumption (i) ensures the superparabolicity of the concerned SPDE, and the boundedness and regularity requirements in (ii) are placed for unique existence of certain regular solutions of SPDE. The readers are referred to \cite{krylov1999analytic} for more discussions. The divergence-free condition (iii) may be thought of as a technical one for the well-posedness of SPDE \eqref{SPDE} (see Remark \ref{rmk-ass-iii});  on the other hand, the common noise in the stochastic integral term $\sigma_t(X_t^{i,\varepsilon})\,dW_t$ induces the fluctuations of the velocity field (of the $i$th particle) formally written as 
	$ v^{i}_t=\sigma_t(X_t^{i,\varepsilon})\frac{dW_t}{dt}$ and in this way, the divergence-free condition means that such fluctuations are of incompressible type.  In fact, such kind of divergence-free conditions have been existing in the literature; refer to \cite{Brze2016Existence,coghi2016propagation} for more clear and elegant arguments. 
\end{rmk}

In the remaining part of the work, we shall use $C$ to denote a generic constant whose value may vary from line to line, and when needed, a bracket will follow immediately after $C$ to indicate what parameters $C$ depend on. 
By $A\hookrightarrow B$ we mean that normed space $(A,\|\cdot\|_{A})$ is embedded into  $(B,\|\cdot\|_B)$ with a constant $C$ such that
$$\|f\|_B\leq C \|f\|_A,\,\,\,\forall f\in A.$$
For readers' convenience, we list Sobolev's embedding theorem in the following lemma, see e.g. \cite[p. 129, p. 131]{Trieb} and \cite[Chapter 9]{brezis2010functional}.
\begin{lem}\label{lem sobolev} 
  There holds the following assertions:

  (i)  For integer $n>d/q+k$ with $k\in \mathbb{N}$ and $q\in(1,\infty)$, we have  $W^{n,q}(\RR^d)\hookrightarrow C^{k,\delta}(\RR^d)$, for any $\delta\in (0,(n-d/q-k)\wedge 1).$

  (ii) If $1<p_0<p_1<\infty$ and $-\infty<s_1<s_0<+\infty$ such that  $\frac{d}{p_0}-s_0=\frac{d}{p_1}-s_1$, then $\mc{H}_{p_0}^{s_0}(\RR^d)\hookrightarrow \mc{H}_{p_1}^{s_1}(\RR^d)$ (with Sobolev spaces as special cases ).

%
\end{lem}



\section{Existence and uniqueness of the solution to SPDE \eqref{SPDE}}
This section is devoted to the global existence and uniqueness of the solution to nonlinear SPDE \eqref{SDE}.

 As already noted in \eqref{cestimate}, if $\rho_t\in L^4$, then it holds that 
\begin{equation}
\norm{c_t}_{W^{2,4}}=\norm{\mc{G}\ast\rho_t}_{W^{2,4}}\leq S_d\norm{\rho_t}_{4}.
\end{equation}
A direct result of Sobolev's embedding theorem implies
\begin{equation}\label{embedding-inf}
\norm{c_t}_{W^{1,\infty}}=\norm{\mc{G}\ast\rho_t}_{W^{1,\infty}}\leq \norm{\mc{G}\ast\rho_t}_{W^{2,4}}\leq S_d\norm{\rho_t}_{4},
\end{equation}
where $S_d$ depends only on $d$.

Before stating the theorem about the well-posedness, we introduce the definition of solutions to SPDE \eqref{SPDE}. Denote by $ C_c^2(\RR^d)$ the space of compactly supported functions having up to second-order continuous derivatives.
\begin{definition}
	A  family of random functions $\{\rho_t(\omega):~t\geq 0,\omega\in \Omega\}$ lying in $S_{\mc F^W}^{\infty}([0,T];L^1\cap L^4(\RR^d))$ is a solution to equation \eqref{SPDE} if
$\rho_t$ satisfies the following stochastic integral equation for all $\varphi\in C_c^2(\RR^d)$,
		\begin{align}
		\la\rho_t,\varphi\ra&=\la\rho_0,\varphi\ra+\chi\int_0^t\la\rho_s,\nabla \varphi  \cdot \nabla c_s\ra\ds+\int_0^t\la\rho_s,\sum_{i=1}^{d}D_i \varphi \sum_{k=1}^{d'}\sigma_s^{ik}\ud W_s^k\ra  \notag\\
		&\quad +\frac{1}{2}\int_0^t\la \rho_s,\sum_{i,j=1}^dD_{ij} \varphi\sum_{k=1}^{d'}(\nu_s^{ik}\nu_s^{jk}+ \sigma_s^{ik}\sigma_s^{jk})\ra\ds\,.
		\end{align}
\end{definition}

\begin{thm}\label{prop-weak-soltn}
Let Assumption \ref{assum1} hold with $m=2$.  Assume $0\leq \rho_0\in L^1\cap \mc{H}^{\frac{1}{2}}_{4}(\RR^d)$\footnote{Here, the initial condition $\rho_0\in \mc{H}^{\frac{1}{2}}_{4}(\RR^d)$ is required by the $L^p$-theory of SPDEs (see \cite[Theorem 5.1]{krylov1999analytic}) for $p=4$.} with $\norm{\rho_0}_1=1$. For each $T>0$, there exists a $\kappa>0$ depending only on $T,\chi,\lambda,\Lambda$ and $ d$ such that if $\|\rho_0\|_{4}\leq \kappa$, SPDE \eqref{SPDE} admits a unique nonnegative   solution in
\begin{equation}\label{defM}
\mathbb M:=L^2_{\mathcal F^W} (0,T;W^{1,2}(\RR^d))\cap L^4_{\mathcal F^W} (0,T;W^{1,4}(\RR^d)) \cap S_{\mathcal F^W}^{\infty}([0,T];L^1\cap L^4(\RR^d)).
\end{equation}
\end{thm}
\begin{proof}
The proof is based on delicate estimates of the solution and the latest developments of $L^p$-theory of SPDE.     First, let
$$
\mathbb B:=\left\{u\in S_{\mathcal F^W}^{\infty} ([0,T];L^4(\RR^d)): \, \|u\|_{S_{\mathcal F^W}^{\infty} ([0,T];L^4(\RR^d))} \leq 
 \ell \kappa  
\right\},
$$
with metric $d(u,v)=\norm{u-v}_{S_{\mathcal F^W}^{\infty} ([0,T]; L^4(\RR^d))}$, and the positive constants $\kappa$ and $\ell $ are to be determined.

Suppose $\|\rho_0\|_{4}\leq \kappa $.  Now we define a map $\mc{T}: \BB\rightarrow S_{\mathcal F^W}^{\infty} ([0,T];L^4(\RR^d))$ as follows:  For each $\xi\in \BB$, let $\mc{T}(\xi):=\rho^\xi$ be the solution to the  following linear SPDE:
\begin{equation}\label{SPDE-linear}
\left\{
\begin{array}{l}
\ud \rho_t
=\left[ \frac{1}{2}\sum_{i,j=1}^dD_{ij}(\rho_t\sum_{k=1}^{d'}(\nu_t^{ik}\nu_t^{jk}+ \sigma_t^{ik}\sigma_t^{jk}))
-\chi\nabla\cdot((\nabla\mc{G}\ast \xi_t) \rho_t)
\right]\dt
-\sum_{i=1}^dD_i( \rho_t \sum_{k=1}^{d'}\sigma_t^{ik}\ud W_t^k),\\
\rho_0\text{ is given}\,.
\end{array}\right.
\end{equation}

 Indeed, as Assumption \ref{assum1} holds with $m=2$, one may write SPDE \eqref{SPDE-linear} as a non-divergence form:
\begin{equation}\label{SPDE-linear-non-divergence-form}
\left\{
\begin{array}{l}
\ud \rho_t
=\left[\frac{1}{2}\sum_{i,j=1}^d\sum_{k=1}^{d'}(\nu_t^{ik}\nu_t^{jk}+ \sigma_t^{ik}\sigma_t^{jk})D_{ij}\rho_t\dt 
+F_t(\rho_t)
\right]\dt-\sum_{i=1}^d \sum_{k=1}^{d'} \sigma_t^{ik} D_i \rho_t\ud W_t^k,\\
\rho_0\text{ is given},
\end{array}\right.
\end{equation}
with 
\begin{align*}
F_t(\rho_t):&=\frac{1}{2}\sum_{i,j=1}^d D_{i}\left(\rho_t\sum_{k=1}^{d'}D_j(\nu_t^{ik}\nu_t^{jk}+ \sigma_t^{ik}\sigma_t^{jk})\right) +\frac{1}{2}\sum_{i,j=1}^d 
D_j\rho_t\sum_{k=1}^{d'}D_i(\nu_t^{ik}\nu_t^{jk}+ \sigma_t^{ik}\sigma_t^{jk}) \\
&\quad
-\chi\nabla\cdot((\nabla\mc{G}\ast \xi_t) \rho_t),
\end{align*} 
where we have used Assumption \ref{assum1} (iii) for the stochastic integral, i.e.,
\begin{align}
\sum_{j=1}^dD_j(\sigma_s^{jk} \rho_s)=\sum_{j=1}^d\rho_s D_j\sigma_s^{jk} +\sum_{j=1}^dD_j\rho_s\sigma_s^{jk} =\sum_{j=1}^dD_j\rho_s \sigma_s^{jk}.
\label{use-ass-iii}
\end{align}
For each $\xi\in \mathbb B$ and $\rho_t\in L^p(\RR^d)$ with $p\in\{2,4\}$, relation \eqref{embedding-inf} indicates that 
\begin{align*}
\| (\nabla\mc{G}\ast\xi_t)\rho_t\|_{p}
&\leq \|\nabla\mc{G}\ast\xi_t\|_{{\infty}} \|\rho_t\|_{p} \leq S_d\norm{\xi_t}_4\|\rho_t\|_{p}\leq \ell  S_d\kappa\|\rho_t\|_{p}, \quad\text{a.s., for all }t\in[0,T].
\end{align*}
This together with Assumption \ref{assum1} allows us, through standard computations, to check that the conditions of the $L^p$-theory of SPDE (see \cite[Theorems 5.1 and 7.1]{krylov1999analytic} for the case when $n=-1$ therein) and the maximum principle (\cite[Theorem 5.12]{krylov1999analytic}) are satisfied and we conclude that the linear SPDE \eqref{SPDE-linear} admits a unique   solution  $\rho^{\xi}$ which is nonnegative and lying in
 $L^p_{\mathcal F^W} (0,T;W^{1,p}(\RR^d)) \cap S^p_{\mathcal F^W} ([0,T];L^p(\RR^d))$, $p\in\{2,4\}$. 

Next we check that $\rho^{\xi}\in S^\infty_{\mathcal F^W}([0,T];L^1\cap L^4(\RR^d))$ and without causing confusion we drop the superscript $\xi$. It is easy to see that the solution of \eqref{SPDE-linear} has the property of conservation of mass, i.e.
\begin{equation*}
\norm{\rho_t}_1=\norm{\rho_0}_1=1\quad \text{a.s.}\,.
\end{equation*}
Applying the It\^o formula for $L^p$-norms in \cite[Theorem 2.1]{krylov2010ito} we have for any $0<t\leq T$
\begin{align}
&\|\rho_{t}\|_{4}^4-\|\rho_0\|_{4}^4   \nonumber\\
=&\int_0^{t } \bigg(\sum_{i,j=1}^d\sum_{k=1}^{d'}-\left\langle 6 |\rho_s|^2 D_i\rho_s,\, D_j\left( (\nu_s^{ik}\nu_s^{jk}+\sigma_s^{ik}\sigma_s^{jk})\rho_s\right) \right\rangle 
+ 6\sum_{k=1}^{d'}\left\langle |\rho_s|^2 ,\, |\sum_{j=1}^dD_j(\rho_s\sigma_s^{jk} )|^2\right\rangle \notag \\
&\quad\quad+12\chi\left\langle \rho_s(\nabla\rho_s),( \nabla\mc{G}\ast\xi_s) \rho_s ^2  \right\rangle \bigg)\,ds 
+12\sum_{i=1}^d\sum_{k=1}^{d'}\int_0^{t}
\left\langle |\rho_s|^2 D_i\rho_s,\,\sigma_s^{ik}\rho_s \right\rangle\,dW^k_s \quad \text{a.s.}\,.  \label{eq-rmk-iii}
\end{align}
Due to (iii) in Assumption \ref{assum1}, we know that  for $k=1,2,\dots,d'$,
$$
12 \sum_{i=1}^d\left\langle |\rho_s|^2 D_i\rho_s,\,\sigma^{ik}\rho_s \right\rangle
= 3\sum_{i=1}^d\left\langle D_i\left (|\rho_s|^4\right) ,\,\sigma^{ik}_s \right\rangle
=-3\left\langle |\rho_s|^4 ,\,\sum_{i=1}^dD_i\sigma^{ik}_s \right\rangle=0\,.
$$
Thus one  has
\begin{align*}
&\|\rho_{t}\|_{4}^4-\|\rho_0\|_{4}^4   \nonumber\\
=&-\int_0^{t} \sum_{i,j=1}^d\sum_{k=1}^{d'}  \left\langle 6 |\rho_s|^2 D_i\rho_s,\, (\nu^{ik}_s\nu_s^{jk}+\sigma_s^{ik}\sigma_s^{jk})D_j\rho_s \right\rangle \,ds
-\sum_{i,j=1}^d\sum_{k=1}^{d'}\int_0^{t}\left\langle 6 |\rho_s|^2 D_i\rho_s,\, D_j(\nu^{ik}_s\nu_s^{jk}+\sigma_s^{ik}\sigma_s^{jk})\rho_s \right\rangle\,ds \nonumber\\
&
+\int_0^{t} 6 \sum_{k=1}^{d'}  \left\langle |\rho_s|^2 ,\, |\sum_{j=1}^dD_j(\rho_s\sigma_s^{jk} ) |^2\right\rangle 
+12\chi\left\langle \rho_s(\nabla\rho_s),( \nabla\mc{G}\ast\xi_s) \rho_s ^2  \right\rangle\,ds \quad \text{a.s.}\,.
\end{align*}
Using (iii) in Assumption \ref{assum1} as in \eqref{use-ass-iii} again
 yields that
\begin{align*}
&-\sum_{i,j=1}^d\sum_{k=1}^{d'}  \left\langle 6 |\rho_s|^2 D_i\rho_s,\, (\nu^{ik}_s\nu_s^{jk}+\sigma_s^{ik}\sigma_s^{jk})D_j\rho_s \right\rangle\\
=&-\sum_{i,j=1}^d\sum_{k=1}^{d'}  \left\langle 6 |\rho_s|^2 D_i\rho_s,\, (\nu^{ik}_s\nu_s^{jk})D_j\rho_s \right\rangle-\sum_{k=1}^{d'}  \left\langle 6 |\rho_s|^2  ,\, |\sum_{j=1}^dD_j\rho_s \sigma_s^{jk}|^2\right\rangle\\
=&-\sum_{i,j=1}^d\sum_{k=1}^{d'}  \left\langle 6 |\rho_s|^2 D_i\rho_s,\, (\nu^{ik}_s\nu_s^{jk})D_j\rho_s \right\rangle-\sum_{k=1}^{d'}  \left\langle 6 |\rho_s|^2  ,\,|\sum_{j=1}^dD_j(\rho_s\sigma_s^{jk} ) |^2\right\rangle\,.
\end{align*}
Therefore it holds that
\begin{align}
&\|\rho_{t}\|_{4}^4-\|\rho_0\|_{4}^4   \nonumber\\
=&-\int_0^{t}  \sum_{i,j=1}^d\sum_{k=1}^{d'}   \left\langle 6 |\rho_s|^2 D_i\rho_s,\, (\nu^{ik}_s\nu_s^{jk})D_j\rho_s \right\rangle \,ds
-\int_0^{t}\sum_{i,j=1}^d\sum_{k=1}^{d'} \left\langle 6 |\rho_s|^2 D_i\rho_s,\, D_j(\nu^{ik}_s\nu_s^{jk}+\sigma_s^{ik}\sigma_s^{jk})\rho_s \right\rangle\,ds \nonumber\\
&
+12 \chi\int_0^{t} 
\left\langle \rho_s(\nabla\rho_s),( \nabla\mc{G}\ast\xi_s) \rho_s ^2  \right\rangle\,ds 
, \quad \text{a.s.}\,. \label{eq:ito-formula}
\end{align}
 
 It follows from $(i)$ in Assumption \ref{assum1} that
 \begin{equation*}
-\sum_{i,j=1}^d\sum_{k=1}^{d'}   \left\langle 6 |\rho_s|^2 D_i\rho_s,\, (\nu^{ik}_s\nu_s^{jk})D_j\rho_s \right\rangle\leq -6\lambda  \|\rho_s\nabla\rho_s\|_{2}^2\,,
 \end{equation*}
 and by $(ii)$ in Assumption \ref{assum1} one has
 \begin{align}
 -6\sum_{i,j=1}^d\sum_{k=1}^{d'} \left\langle  |\rho_s|^2 D_i\rho_s,\, D_j(\nu^{ik}_s\nu_s^{jk}+\sigma_s^{ik}\sigma_s^{jk})\rho_s \right\rangle
 &\leq 24\Lambda^2 
 \left\langle  |\rho_s|^2 |\nabla \rho_s|,\, \rho_s \right\rangle
  \label{x-q-est}
   \\
 &\leq  2\lambda\|\rho_s\nabla\rho_s\|_{2}^2+\frac{(12\Lambda^2 )^2}{2\lambda}\|\rho_s\|_{4}^4
 \,.\nonumber
 \end{align}
 We also notice that
\begin{align}
12\chi\left\langle \rho_s(\nabla\rho_s),( \nabla\mc{G}\ast\xi_s) \rho_s ^2  \right\rangle
&\leq
12\chi \|\rho_s\nabla\rho_s\|_{2} \|\rho_s\|_{4}^2 \|\nabla\mc{G}\ast\xi_s\|_{{\infty}}     \nonumber\\
\text{(by relation \eqref{embedding-inf})}&\leq
12\chi S_d \|\rho_s\nabla\rho_s\|_{2} \|\rho_s\|_{4}^{2}\norm{\xi_s}_4\leq
12\ell\chi S_d\kappa\|\rho_s\nabla\rho_s\|_{2} \|\rho_s\|_{4}^{2}  \nonumber \\
\text{(by Young's inequality)} 
&\leq
2\lambda  \|\rho_s\nabla \rho_s\|_{2}^2 + \frac{(6\ell\chi S_d\kappa)^2}{2\lambda}\|\rho_s\|_{4}^{4} \,.\label{est-F-11}
\end{align}
Collecting above estimates, \eqref{eq:ito-formula} yields that
\begin{align}
&\|\rho_{t}\|_{4}^4-\|\rho_0\|_{4}^4\notag\\
\leq&
-6\lambda  \int_0^{t} \|\rho_s\nabla\rho_s\|_{2}^2\,ds 
+4\lambda\int_0^{t}\|\rho_s\nabla\rho_s\|_{2}^2\,ds 
+
 \left(\frac{(12\Lambda^2)^2}{2\lambda}+\frac{(6\ell\chi S_d\kappa)^2}{2\lambda}\right) 
 \int_0^{t} \|\rho_s\|_{4}^{4} \,ds \notag\\
 \leq&  \left(\frac{(12\Lambda^2)^2}{2\lambda}+\frac{(6\ell\chi S_d\kappa)^2}{\lambda}\right) 
 \int_0^{t} \|\rho_s\|_{4}^{4} \,ds\,. \label{eq-ell-ell}
\end{align}
 Take a sufficiently large $\ell>1$ and relatively small $\kappa_0$\footnote{The selections of $\ell$ and $ \kappa$ are not unique; a particular case  is to take $\kappa_0\leq \frac{1}{\chi \ell}$ with 
	$$
	\ell= \exp\left\{
	\frac{T}{4} \left(\frac{(12\Lambda^2)^2}{2\lambda}+\frac{(6 S_d)^2}{\lambda}\right) 
	\right\}.
	$$} such that whenever $\kappa\leq \kappa_0$ it holds that
\begin{align}
\exp\left\{
\frac{T}{4} \left(\frac{(12\Lambda^2)^2}{2\lambda}+\frac{(6\ell\chi S_d\kappa)^2}{\lambda}\right) 
\right\}
\leq \ell.\label{select-ell}
\end{align}
Applying Gronwall's inequality to \eqref{eq-ell-ell} yields that
\begin{align*}
\sup_{t\in[0,T]} \|\rho_{t}\|_{4} &\leq \|\rho_0\|_{4} \exp\left\{
\frac{T}{4} \left(\frac{(12\Lambda^2)^2}{2\lambda}+\frac{(6\ell\chi S_d\kappa)^2}{\lambda}\right) 
\right\},\\
&\leq \ell \kappa,
\end{align*}
which gives that $\rho\in \mathbb B$.  

Fix the constants $\ell$ and $\kappa_0$ as selected above. Let $\kappa\leq \kappa_0$. For all $\xi\in\BB$, let $\rho^{\xi}$ be the unique solution of the linear SPDE \eqref{SPDE-linear}. From the discussion above, we get the solution map
$$
\mathcal T:\quad  \mathbb B \rightarrow \mathbb B, \quad \xi\mapsto \rho^{\xi}.
$$
Next we show that the map $\mc{T}$ is a contraction.

For any $\bar\xi,\, \hat\xi\in \mathbb B$, set $\delta \rho=\rho^{\bar\xi}-\rho^{\hat\xi}$ and $\delta \xi=\bar\xi-\hat\xi$. As before, we apply It\^o formula for the $L^4$-norm of $\delta \rho$:
\begin{align}	
	&\|\delta\rho_{t}\|_{4}^4   \nonumber\\
	=&
	\int_0^{t}\left(-\sum_{i,j=1}^d\sum_{k=1}^{d'}\left\langle 6 |\delta\rho_s|^2 D_i\delta\rho_s,\, D_j\left((\nu_s^{ik}\nu_s^{jk}+\sigma_s^{ik}\sigma_s^{jk})\delta\rho_s\right) \right\rangle 
	+ 6\sum_{k=1}^{d'}\left\langle |\delta\rho_s|^2 ,\, |\sum_jD_j(\delta\rho_s\sigma_s^{jk} )|^2\right\rangle \right)\,ds
	\nonumber\\
	&
	+\int_0^t12\chi\left\langle |\delta\rho_s|^2  \nabla\delta\rho_s, \nabla\mc{G}\ast \bar\xi_s \rho^{\bar\xi}_s- \nabla\mc{G}\ast \hat\xi_s \rho^{\hat\xi}_s  \right\rangle\,ds 
	+
	12\sum_{i=1}^d\sum_{k=1}^{d'}\int_0^{t}
	\left\langle |\delta \rho_s|^2 D_i\delta\rho_s,\,\sigma_s^{ik}\delta\rho_s \right\rangle\,dW^k_s
	\nonumber\\
	=&
	\int_0^{t}\left(-\sum_{i,j=1}^d\sum_{k=1}^{d'}  \left\langle 6 |\delta\rho_s|^2 D_i\delta\rho_s,\, (\nu_s^{ik}\nu_s^{jk})D_j\delta\rho_s\right\rangle 
	-\sum_{i,j=1}^d\sum_{k=1}^{d'} \left\langle 6 |\delta\rho_s|^2 D_i\delta\rho_s,\, D_j(\nu^{ik}_s\nu_s^{jk}+\sigma_s^{ik}\sigma_s^{jk})\delta\rho_s \right\rangle \right)\,ds
	\nonumber\\
	&
	+    \int_0^t12\chi\left\langle |\delta\rho_s|^2  \nabla \delta\rho_s, \nabla\mc{G}\ast \bar\xi_s \rho^{\bar\xi}_s- \nabla\mc{G}\ast \hat\xi_s \rho^{\hat\xi}_s  \right\rangle\,ds
	\nonumber\\
	\leq&
	-6\lambda  \int_0^{t} \|\delta \rho_s\nabla\delta\rho_s\|_2^2\,ds 
	+2\lambda\int_0^t\|\delta\rho_s\nabla\delta\rho_s\|_{2}^2\,ds+\frac{(12\Lambda^2 )^2}{2\lambda}\int_0^t\|\delta\rho_s\|_{4}^4\,ds
	\label{eq:ito-delta}
	\\
	&\quad
	+\int_0^t12\chi\left\langle |\delta\rho_s|^2  \nabla \delta\rho_s, \nabla\mc{G}\ast \bar\xi_s \rho^{\bar\xi}_s- \nabla\mc{G}\ast \hat\xi_s \rho^{\hat\xi}_s  \right\rangle\,ds,\quad\text{a.s.}. 
	\nonumber
\end{align}
Let us compute that
\begin{align*}
12\chi\left\langle |\delta\rho_s|^2  \nabla \delta\rho_s, \nabla\mc{G}\ast \delta\xi_s \rho^{\bar\xi}_s   \right\rangle
&
\leq 12\chi\|\delta\rho_s \nabla\delta\rho_s\|_{2}  \|  \delta\rho_s\rho^{\bar\xi}_s\|_{2} \| \nabla\mc{G}\ast\delta\xi_s  \|_{{\infty}}\\
&\leq 
12 \chi S_d\|\delta\rho_s \nabla\delta\rho_s\|_{2}  \|  \delta\rho_s\|_{4}\|\rho^{\bar\xi}_s\|_{4}\norm{ \delta\xi_s }_4
\\
&\leq 2\lambda\|\delta\rho_s \nabla\delta\rho_s\|_{2}^2+\frac{(6\chi S_d)^2}{2\lambda}\|  \delta\rho_s\|_{4}^2\|\rho^{\bar\xi}_s\|_{4}^2\norm{ \delta\xi_s }_4^2
\\
&\leq 
	2\lambda\|\delta\rho_s \nabla\delta\rho_s\|_{2}^2+\frac{(6\chi S_d)^2}{4\lambda}
\left((\ell\kappa)^{2}\| \delta\rho_s\|_{4}^4+(\ell\kappa)^{-2}\|\rho^{\bar\xi}_s\|_{4}^4\norm{ \delta\xi_s }_4^4\right)
\\
&\leq
 2\lambda\|\delta\rho_s \nabla\delta\rho_s\|_{2}^2+\frac{(6\ell\chi S_d\kappa)^2}{4\lambda}\| \delta\rho_s\|_{4}^4+\frac{(6\ell\chi S_d\kappa)^2}{4\lambda}\| \delta\xi_s\|_{4}^4\,.
\end{align*}
In a similar way to \eqref{est-F-11}, we have
\begin{align*}
12\chi\left\langle |\delta\rho_s|^2  \nabla \delta\rho_s,  \nabla\mc{G}\ast\hat\xi_s \delta\rho_s  \right\rangle
\leq
2\lambda  \|\delta\rho_s\nabla \delta\rho_s\|_{2}^2 + \frac{(6\ell\chi S_d\kappa)^2}{2\lambda}\|\delta\rho_s\|_{4}^4\,.
\end{align*}
Thus, combining above estimates gives
\begin{align}\label{314}
&12\chi\left\langle |\delta\rho_s|^2  \nabla \delta\rho_s, \nabla\mc{G}\ast \bar\xi_s \rho^{\bar\xi}_s- \nabla\mc{G}\ast \hat\xi_s \rho^{\hat\xi}_s  \right\rangle
=
12\chi\left\langle |\delta\rho_s|^2  \nabla \delta\rho_s, \nabla\mc{G}\ast\delta\xi_s \rho^{\bar\xi}_s+ \nabla\mc{G}\ast \hat\xi_s \delta\rho_s  \right\rangle \notag\\
\leq&
 4\lambda  \|\delta\rho_s\nabla \delta\rho_s\|_{2}^2 + \frac{(6\ell\chi S_d\kappa)^2}{\lambda}\|\delta\rho_s\|_{4}^4+ \frac{(6\ell\chi S_d\kappa)^2}{4\lambda}\|\delta\xi_s\|_{4}^4\,,
\end{align}
which together with \eqref{eq:ito-delta} and \eqref{select-ell} implies
\begin{align*}
\|\delta\rho_{t}\|_{4}^4 
&\leq
\left(\frac{(12\Lambda^2 )^2}{2\lambda}+\frac{(6\ell\chi S_d\kappa)^2}{\lambda}\right)	\int_0^t\|\delta\rho_s\|_{4}^4 \,ds +\frac{(6\ell\chi S_d\kappa)^2}{4\lambda} \int_0^t \|\delta\xi_s\|_{4}^4\,ds\\
&\leq \frac{4\ln \ell}{T}\int_0^t\|\delta\rho_s\|_{4}^4 \,ds +\frac{(6\ell\chi S_d\kappa)^2}{4\lambda} \int_0^t \|\delta\xi_s\|_{4}^4\,ds\quad \text{a.s.}\,.
\end{align*}
 By Gronwall's inequality, we get
\begin{align}
\|\delta \rho\|_{S_{\mc F^W}^{\infty}([0,T];L^4(\RR^d))}  
& \leq \left[\frac{(6\ell\chi S_d\kappa)^2T}{4\lambda}\right]^{\frac{1}{4}}\ell\|\delta \xi\|_{S_{\mc F^W}^{\infty}([0,T];L^4(\RR^d))}\quad \text{a.s.}\,.
\label{q-est-thm}
\end{align}
 Hence, whenever  $0<\kappa<\kappa_0 \wedge \left[\frac{4\lambda}{(6\ell^3\chi S_d)^2T}\right]^{\frac{1}{2}}$, the solution map $\mathcal T$ is a contraction mapping on the complete metric space $\mathbb B$, and it admits a unique fixed point $\rho=\rho^{\rho}$ which is the unique   solution to SPDE \eqref{SPDE}.
\end{proof}

\begin{rmk}\label{rmk-ass-iii}
For the well-posedness of SPDE \eqref{SPDE}, the main difficulty lies in the KS type nonlinear term $-\chi\nabla\cdot((\nabla\mc{G}\ast \rho_t) \rho_t)$ which prevents us from using the existing methods in the SPDE literature. In view of equation \eqref{eq-rmk-iii} and the computation that follows, one may see that the stochastic integral there equals zero because of  the divergence-free condition  (iii) of Assumption \ref{assum1}. This further allows us to obtain $\rho\in S_{\mathcal F^W}^\infty(0,T;L^4(\RR^d))$ which finally yields the conclusions in Theorem \ref{prop-weak-soltn} with a \textit{deterministic} $\kappa$. Without (iii) of Assumption \ref{assum1}, one may try to generalize the localization technique with stopping times  (see \cite[Chapter 1, Section 5]{KaratzasShreve1998}) for random fields which, however, may incur cumbersome arguments not just for the well-posedness of SPDE \eqref{SPDE} in this section, but also for the subsequent sections.
\end{rmk}

In view of the above proof, we can particularly take
\begin{align*}
0< \chi\kappa<\frac{1}{ \ell} \wedge  \left[\frac{4\lambda}{(6\ell^3 S_d)^2T}\right]^{\frac{1}{2}},\quad \text{with }	\ell= \exp\left\{
	\frac{T}{4} \left(\frac{(12\Lambda^2)^2}{2\lambda}+\frac{(6 S_d)^2}{\lambda}\right) 
	\right\},
\end{align*}
for the well-posedness of SPDE \eqref{SPDE} in Theorem \ref{prop-weak-soltn}. Therefore, whenever $ \chi\|\rho_0\|_{4}<\frac{1}{ \ell} \wedge  \left[\frac{4\lambda}{(6\ell^3 S_d)^2T}\right]^{\frac{1}{2}}$, the unique existence of solution in $\mathbb M$ can be asserted as in Theorem \ref{prop-weak-soltn}. 

Furthermore, suppose that the diffusion coefficients $\nu$ and $\sigma$ are spatial invariant, i.e.,
\begin{align}
\text{the measurable diffusion coefficients } 
\sigma,\,\nu: \left([0,T],\mathcal B([0,T]\right) \longrightarrow \left(\RR^{d\times d'},\mathcal B(\RR^{d\times d'})\right). \label{spatial-invariant}
\end{align}
Then the left-hand side of \eqref{x-q-est} and the third term of line \eqref{eq:ito-delta} will vanish. Repeating the proof and combining computations around \eqref{select-ell} and \eqref{q-est-thm}, we can obtain the well-posedness of SPDE \eqref{SPDE} in Theorem \ref{prop-weak-soltn} with a particular selection:
\begin{align}
0<\chi \kappa \sqrt{T}<\frac{1}{\ell} \wedge  \left[\frac{4\lambda}{(6\ell^3 S_d)^2}\right]^{\frac{1}{2}},\quad \text{with }	\ell= \exp\left\{
	\frac{(6 S_d)^2}{ 4 \lambda}
	\right\},
	\label{xx-est-good}
\end{align}
which indicates that for any given $\rho_0$, the existence and uniqueness of solution may be guaranteed on time interval $[0,T_0]$ if
$$
T_0< \frac{1}{ \ell^2 \chi ^2 \|\rho_0\|_4^2} \wedge  \frac{4\lambda}{(6\ell^3\chi \|\rho_0\|_4 S_d)^2}.
$$
For this solution on $[0,T_0]$, we may conduct estimates as in the proof of Theorem \ref{prop-weak-soltn}.
Notice that instead of \eqref{est-F-11} and \eqref{eq-ell-ell}, we have
\begin{align*}
12\chi\left\langle \rho_s(\nabla\rho_s),( \nabla\mc{G}\ast\rho_s) \rho_s ^2  \right\rangle
&\leq
12\chi \|\rho_s\nabla\rho_s\|_{2} \|\rho_s\|_{4}^2 \|\nabla\mc{G}\ast\rho_s\|_{{\infty}}     \nonumber\\
 &\leq
12\chi S_d \|\rho_s\nabla\rho_s\|_{2} \|\rho_s\|_{4}^{2}\norm{\rho_s}_4
  \nonumber \\
&\leq
 2\lambda  \|\rho_s\nabla \rho_s\|_{2}^2 + \frac{(6\chi S_d)^2}{2\lambda}\|\rho_s\|_{4}^{6},
\end{align*}
and
\begin{align}
&\|\rho_{t}\|_{4}^4-\|\rho_0\|_{4}^4\notag\\
\leq&
-6\lambda  \int_0^{t} \|\rho_s\nabla\rho_s\|_{2}^2\,ds 
+2\lambda\int_0^{t}\|\rho_s\nabla\rho_s\|_{2}^2\,ds 
+ \frac{(6 \chi S_d)^2}{2\lambda} 
 \int_0^{t} \|\rho_s\|_{4}^{4} \,ds \notag\\
 \leq& -{\lambda}  \int_0^{t} \|\nabla(\rho_s^2)\|_{2}^2\,ds 
 + \frac{(6\chi S_d)^2}{2\lambda} 
 \int_0^{t} \|\rho_s\|_{4}^{6} \,ds\,. \label{xx-q-est-rho}
\end{align}
Meanwhile, using the Gagliardo-Nirenberg inequality yields that there exists a constant $N_d>0$ depending on $d$ such that
\begin{align*}
\|\rho_s\|_4^2=\|\rho^2_s\|_2 
\leq N_d \|\nabla (\rho^2_s)\|_2^{\frac{d}{d+2}}\cdot \|\rho_s^2\|^{\frac{2}{d+2}}_1
&=
 N_d \|\nabla (\rho^2_s)\|_2^{\frac{d}{d+2}}\cdot \|\rho_s\|^{\frac{4}{d+2}}_2
 \\
 \text{(by interpolation inequality) }&\leq
 N_d \|\nabla (\rho^2_s)\|_2^{\frac{d}{d+2}}\cdot \|\rho_s\|^{\frac{8}{3(d+2)}}_4 
 \cdot \|\rho_s\|^{\frac{4}{3(d+2)}}_1\\
 &=N_d \|\nabla (\rho^2_s)\|_2^{\frac{d}{d+2}}\cdot \|\rho_s\|^{\frac{8}{3(d+2)}}_4.
\end{align*}
Then it follows that 
\begin{align*}
\|\nabla (\rho^2_s)\|^2_2
&\geq 
|N_d|^{-\frac{2(d+2)}{d}} \cdot  \|\rho_s\|_4^{4+\frac{8}{3d}}, 
\end{align*}
which inserted into \eqref{xx-q-est-rho} gives
 \begin{align}
\|\rho_{t}\|_{4}^4-\|\rho_0\|_{4}^4
 \leq& \int_0^{t}  \|\rho_s\|_4^{4+\frac{8}{3d}} \left(
    \frac{(6\chi S_d)^2}{2\lambda}  \cdot  \|\rho_s\|_4^{2-\frac{8}{3d}}  
    -\frac{\lambda }{ |N_d|^{\frac{2(d+2)}{d}}} 
 \right)
 \,ds\,.   \label{xx-q-est-global}
\end{align}
Therefore, if 
$$
 \frac{(6\chi S_d)^2}{2\lambda}  \cdot  \|\rho_0\|_4^{2-\frac{8}{3d}}  
    -\frac{\lambda }{ |N_d|^{\frac{2(d+2)}{d}}} \leq 0
$$
i.e.,
\begin{align}
\|\rho_0\|_4 \leq \left(
\frac{2\lambda^2}{ |N_d|^{\frac{2(d+2)}{d}}(6\chi S_d)^2    }
\right)^{\frac{3d}{6d-8}}, \label{xx-q-est-global-condtn}
\end{align}
then we conclude from \eqref{xx-q-est-global} that $\|\rho_t\|_4 \leq \|\rho_0\|_4$ for all $t\in [0,T_0]$ and that the unique solution may actually be extended to any finite time interval $[0,T]$. 
\begin{cor}\label{cor-global-solution}
Let Assumption \ref{assum1} hold with $m=2$ and the diffusion coefficients $\nu$ and $\sigma$ being spatial invariant (see \eqref{spatial-invariant}).  Assume $0\leq \rho_0\in L^1\cap \mc{H}^{\frac{1}{2}}_{4}(\RR^d)$ with $\norm{\rho_0}_1=1$. There exists a constant $\kappa>0$ depending only on $\chi,\lambda,$ and $ d$ such that if $\|\rho_0\|_{4}\leq \kappa$, SPDE \eqref{SPDE} admits a unique nonnegative   solution in
\begin{equation*}
L^2_{\mathcal F^W} (0,T;W^{1,2}(\RR^d))\cap L^4_{\mathcal F^W} (0,T;W^{1,4}(\RR^d)) \cap S_{\mathcal F^W}^{\infty}([0,T];L^1\cap L^4(\RR^d)), 
\end{equation*}
for all $T>0$.
\end{cor}
In  Corollary \ref{cor-global-solution}, the constant $\kappa$ may be given as the right-hand side of \eqref{xx-q-est-global-condtn} that is independent of $(T,\Lambda)$ and the global solution result with small initial value under $L^4$-norm seems to hold in a similar way as the deterministic counterparts (see \cite{blanchet2006two,corrias2004global,biler2010blowup} for instance). The results in Theorem \ref{prop-weak-soltn}, Corollary \ref{cor-global-solution}, and subsequent theorems,  may be  extended to general $L^p$-norms for $p>3$, which would not be discussed in this paper to avoid cumbersome arguments.

To explore the connections between the stochastic Keller-Segel equation \eqref{SPDE} and associated SDEs  of McKean-Vlasov type \eqref{SDE}, we need stronger regularity of the solution.


\begin{thm}\label{thm-wellposedness}
Let Assumption \ref{assum1} hold with $m=3$. Suppose further $\rho_0 \in L^1\cap W^{2,2}(\RR^d)$. 
Then for any $T>0$, there exists $\kappa>0$ depending only on $T,\Lambda,\lambda,\chi$ and $ d$ such that if $\|\rho_0\|_{4}\leq \kappa$, SPDE \eqref{SPDE} admits a unique nonnegative solution in
$$\mathbb M_1:=L^2_{\mathcal F^W} (0,T;W^{3,2}(\RR^d)) \cap S_{\mc F^W}^2 ([0,T];W^{2,2}(\RR^d)) \cap L^4_{\mathcal F^W} (0,T;W^{1,4}(\RR^d)) \cap S_{\mc F^W}^{\infty}([0,T];L^1\cap L^4(\RR^d)).$$
\end{thm}
\begin{proof}
Notice that $W^{2,2}(\RR^d)\hookrightarrow \mc{H}_4^{\frac{1}{2}} (\RR^d)\hookrightarrow L^4(\RR^d)$ for $d=2$ or $3$. Comparing Theorem \ref{thm-wellposedness} and Theorem \ref{prop-weak-soltn}, we only need to prove that the obtained unique   solution $\rho$ in Theorem \ref{prop-weak-soltn} is also lying in $L^2_{\mathcal F^W} (0,T;W^{3,2}(\RR^d)) \cap S_{\mc F^W}^2 ([0,T];W^{2,2}(\RR^d))$. In fact, $\rho\in \mathbb M$ (defined in \eqref{defM}) is the   solution of the following linear SPDE:
\begin{align}\label{SPDE-linear-1}
\begin{cases}
\ud \rho_t
=\left[\frac{1}{2}\sum_{i,j=1}^dD_{ij}(\rho_t\sum_{k=1}^{d'}(\nu^{ik}\nu^{jk}+ \sigma^{ik}\sigma^{jk}))+\chi f_t \right]\dt-\sum_{i=1}^dD_i( \rho_t\sum_{k=1}^{d'} \sigma^{ik})\ud W_t^k \\
\rho_0\text{ is given,}
\end{cases}
\end{align}
with 
$$f_t=-\nabla\cdot(\rho_t\nabla c_t)=-\nabla\rho_t\cdot\nabla c_t+\rho_t^2-\rho_tc_t.$$
As $\rho\in \mathbb M$, it follows that
\begin{align*}
\|f_t\|_{2}&=\|\nabla\cdot(\rho_t\nabla c_t)\|_{2} \leq \norm{\nabla c_t}_\infty\norm{\nabla \rho_t}_{2}+\|\rho_t \|_{4}^2+\norm{\rho_t}_2\norm{c_t}_\infty
\nonumber\\
&\leq S_d \|\rho\|_{4} \norm{\nabla\rho_t}_{2}+\|\rho_t \|_{4}^2+S_d\norm{\rho_t}_2\norm{\rho_t}_4\leq S_d\|\rho\|_{4} \norm{\rho_t}_{W^{1,2}}+\|\rho_t \|_{4}^2
\,,
\end{align*}
which indicates that
\begin{align}\label{est-f-1}
\|f\|^2_{L^2_{\mathcal F^W}(0,T;L^2)}
\leq 
2S_d^2 \|\rho\|^2_{S_{\mc F^W}^{\infty}([0,T];L^4)}  \| \rho\|^2_{L^2_{\mathcal F^W}(0,T;W^{1,2})}+ 2\| \rho\|^4_{L^4_{\mathcal F^W}(0,T;L^4)}<\infty\,.
\end{align}
The $L^p$-theory of SPDE (see  \cite[Theorem 5.1]{krylov1999analytic}) and Theorem \ref{prop-weak-soltn} imply that 
\begin{align}
\rho\in L^2_{\mathcal F^W} (0,T;W^{2,2}(\RR^d)) \cap S_{\mc F^W}^2 ([0,T];W^{1,2}(\RR^d)) \cap \mathbb M. \label{est-regular-1}
\end{align}

Similarly, for $j=1,\dots,d$, one has
\begin{align*}
&\|D_jf_t\|_{2}\leq C\norm{\rho_t}_{W^{2,2}}\norm{\rho_t}_4+C\norm{\rho_t}_{W^{1,4}}\norm{\rho_t}_4\,,
\end{align*}
which together with \eqref{est-regular-1} and \eqref{est-f-1} implies that
$$
\|f\|_{L^2_{\mathcal F^W}(0,T;W^{1,2})} <\infty.
$$
Hence, applying the $L^p$-theory of SPDE (see  \cite[Theorem 5.1]{krylov1999analytic}) and Theorem \ref{prop-weak-soltn} again, we conclude
$$\rho\in L^2_{\mathcal F^W} (0,T;W^{3,2}(\RR^d)) \cap S_{\mc F^W}^2 ([0,T];W^{2,2}(\RR^d)) \cap L^4_{\mathcal F^W} (0,T;W^{1,4}(\RR^d)) \cap S_{\mc F^W}^{\infty}([0,T];L^1\cap L^4(\RR^d)).$$
The proof is completed.
\end{proof}


\section{Well-posedness of the nonlinear SDE}
Let us consider the following SDE:
\begin{align}\label{SDE1}
\begin{cases}
\ud Y_t=\chi \nabla \mc{G}\ast\rho_t(Y_t)\dt+\nu_t(Y_t) \ud B_t+\sigma_t(Y_t)\ud W_t,\quad t>0\,,\\
\rho_t\mbox{ is the conditional density of  }Y_t \mbox{ given } \mc{F}_t^W\,,\\
Y_0=\zeta^1,
\end{cases}
\end{align}
where we take $B_t=B_t^1$ \textit{in this section} as a $d'$-dimensional Wiener process independent of $W_t$ and $\zeta^1$. In the following, we prove the well-posedness of the nonlinear SDE \eqref{SDE1} which actually shares the same solvability as SDE \eqref{SDE} for each $i\in \mathbb Z^+$.
\begin{thm}\label{thmSDE}(Well-posedness of the SDE)
Under the same assumptions as in Theorem \ref{thm-wellposedness}, let $\rho  $ be the regular   solution to the SPDE \eqref{SPDE} obtained in Theorem \ref{thm-wellposedness}. 
Then the nonlinear SDE \eqref{SDE1} has a unique strong solution $(Y_t)_{t\geq 0}$  with $\rho\in  S^{2}_{\mc F^W}([0,T];W^{2,2}(\RR^d)) \cap S^{\infty}_{\mc F^W}([0,T];L^{4}(\RR^d))$ being its conditional density under filtration $(\mc F^W_t)_{t\in[0,T]}$.
\end{thm}
\begin{proof}
For the solution $\rho\in S^{2}_{\mc F^W}(([0,T];W^{2,2}(\RR^d)) \cap S^{\infty}_{\mc F^W}([0,T];L^4(\RR^d))$ of the SPDE \eqref{SPDE} given in Theorem  \ref{thm-wellposedness},  by embedding theorems , we have
\begin{align}
\nabla\mc{G}\ast\rho \in S_{\mc F^W}^2([0,T];W^{3,2}(\RR^d)) \cap S^{\infty}_{\mc F^W}([0,T];W^{1,4}(\RR^d)) \,\, \hookrightarrow\,\,S_{\mc F^W}^2([0,T];W^{1,\infty}(\RR^d)) \cap   S^{\infty}_{\mc F^W}([0,T];L^{\infty}(\RR^d))\,,
\label{relatn-G}
\end{align}
which ensures the existence and uniqueness of strong solution $(\overline Y_t)_{t\geq 0}$ to the following linear SDE:
\begin{align}\label{SDElinear}
\begin{cases}
\ud \overline Y_t=\chi \nabla\mc{G} \ast\rho_t(\overline Y_t)\dt+\nu_t(\overline Y_t) \ud B_t+\sigma_t(\overline Y_t)\ud W_t,\quad t>0\,,\\
\overline Y_0 =\zeta^1.
\end{cases}
\end{align}
{To prove that the  conditional density given  $\mc{F}_t^W$  of $(\overline Y_t)_{t\geq 0}$ exists and is the solution to SPDE \eqref{SPDE}, we need the following result on backward SPDE and associated probabilistic representation.
\begin{lem}\label{lem-BSPDE}
Let Assumption \ref{assum1} hold with $m=3$, $\rho\in  S^{2}_{\mc F^W}([0,T];W^{2,2}(\RR^d)) \cap S^{\infty}_{\mc F^W}([0,T];L^{4}(\RR^d)) $ and $T_1\in(0,T]$. Then for each $G\in L^2(\Omega, \mc F_{T_1}; W^{2,2}(\RR^d))$, the following backward SPDE:
\begin{equation}\label{BSPDE}
\left\{
\begin{split}
-\ud u_t
&=\left[ \frac{1}{2}\sum_{i,j=1}^d\sum_{k=1}^{d'}(\nu^{ik}_t\nu^{j,k}_t+ \sigma^{ik}_t\sigma^{jk}_t)D_{ij}u_t +\chi\sum_{i=1}^d D_i\mc{G} \ast\rho_t D_iu_t  + \sum_{i=1}^d\sum_{k=1}^{d'} \sigma_t^{ik}D_i \psi_t^k  \right]\,dt   \\
&\quad-\sum_{k=1}^{d'} \psi_t^k \ud W_t^k, \\
u_{T_1}&=G ,
\end{split}
\right.
\end{equation}
admits a unique solution 
$$(u,\psi)  \in \left( L^2_{\mathcal F^W} (0,T;W^{3,2}(\RR^d)) \cap S_{\mc F^W}^2 ([0,T];W^{2,2}(\RR^d)) \right) \times L^2_{\mathcal F^W} (0,T;W^{2,2}(\RR^d)) ,  $$
i.e., for any $\varphi\in C_c^2(\RR^d)$, there holds for each $t\in[0,T_1]$,
\begin{align*}
\langle u_t, \,\varphi\rangle
&= \langle \varphi,\,G\rangle
+\int_t^{T_1} \bigg\langle \varphi,\,  \frac{1}{2}\sum_{i,j=1}^d\sum_{k=1}^{d'}(\nu_s^{ik}\nu_s^{jk}+ \sigma_s^{ik}\sigma_s^{jk})D_{ij}u_s 
		+\chi\sum_{i=1}^d D_i\mc{G} \ast\rho_s D_iu_s  + \sum_{i=1}^d\sum_{k=1}^{d'} \sigma_s^{ik}D_i \psi_s^k \bigg\rangle\,ds
		\\
&\quad\quad	
	-\int_t^{T_1} \sum_{k=1}^{d'} \left\langle \varphi,\,  \psi^k_s\right\rangle \,dW_s^k ,\quad \text{a.s.}
\end{align*}
Moreover, for this solution, we have 
\begin{align}
u_t(y)=\EE\left[  G(\overline Y_{T_1}) \big|\overline Y_t=y,\, \mc F_t^W\right],\quad \text{a.s. for any }t\in[0,T_1].  \label{eq-prob-rep}
\end{align}
\end{lem}

For each $T_1\in (0,T]$, take an arbitrary $\xi\in L^{\infty}(\Omega,\mc F_{T_1})$ and $\phi\in C_c^{\infty}(\RR^d)$. In view of the SPDE \eqref{SPDE}, applying the It\^o formula to $\langle u_t,\,\rho_t\rangle $ (the duality analysis on the \eqref{SPDE} and \eqref{BSPDE} as in \cite{DuQiuTang10,Zhou_92})) gives
\begin{align*}
\langle u_0, \rho_0\rangle = \langle \xi\phi, \rho_{T_1} \rangle
-\int_0^{T_1} \sum_{i=1}^d\sum_{k=1}^{d'} \langle u_t,\,D_i(\sigma^{ik}_t\rho_t)\rangle\, dW^k_t -\int_0^{T_1} \sum_{k=1}^{d'} \langle \rho_t,\, \psi^k_t\rangle\, dW^k_t, \quad \text{a.s.,}
\end{align*}
where $(u,\psi)$ is the solution in Lemma \ref{lem-BSPDE} with $G=\xi\phi$. Then we have by taking expectations on both sides,
\begin{align*}
\langle u_0, \rho_0\rangle = \EE [\langle \xi\phi, \rho_{T_1} \rangle] = \EE [\xi \langle \phi, \rho_{T_1} \rangle]  .
\end{align*}
On the other hand, in view of the probabilistic representation \eqref{eq-prob-rep}, we have
\begin{align*}
\langle u_0, \rho_0\rangle =\int_{\RR^d}\EE\left[  G(\overline Y_{T_1}) \big|\overline Y_0=y,\, \mc F_0^W\right]\rho_0(y)\dy=\EE [\xi\phi(\overline Y_{T_1}) ]= \EE\left[  \xi \EE[\phi(\overline Y_{T_1}) \big| \mc F^W_{T_1}]   \right].
\end{align*}
Therefore, 
$$
\EE [\xi \langle \phi, \rho_{T_1} \rangle] 
= \EE\left[  \xi \EE[\phi(\overline Y_{T_1}) \big| \mc F^W_{T_1}]   \right],
$$
which by the arbitrariness of $(T_1,\xi,\phi)$ implies that $\rho_t$ is the conditional density of $\overline Y_t $ given  $\mc{F}_t^W$  for each $t\in[0,T]$, and shows the existence of strong solution to SDE \eqref{SDE1}.  In fact, this also means that each strong solution of SDE \eqref{SDE1} with $\rho\in S^{2}_{\mc F^W}([0,T];W^{2,2}(\RR^d)) \cap S^{\infty}_{\mc F^W}([0,T];L^{4}(\RR^d))$ must have the conditional density $\rho$ being the solution to SPDE \eqref{SPDE}, and thus, the strong solution  is unique. We complete the proof.
}
\end{proof}

\begin{proof}[Proof of Lemma \ref{lem-BSPDE}]
Embedding theorem gives \eqref{relatn-G} which by the $L^2$-theory of backward SPDE (see \cite{DuQiuTang10,Zhou_92}) implies that backward SPDE \eqref{BSPDE} has a unique solution $(u,\psi)\in   \left( L^2_{\mathcal F^W} (0,T;W^{1,2}(\RR^d)) \cap S_{\mc F^W}^2 ([0,T];L^2(\RR^d)) \right) \times L^2_{\mathcal F^W} (0,T;L^2(\RR^d))$.\footnote{The fact 
$u\in S_{\mc F^W}^2 ([0,T];L^2(\RR^d))$ is not claimed in \cite{DuQiuTang10,Zhou_92} but it follows straightforwardly from \cite[Theorem A.2]{RRW_2007} for It\^o's formula of square norms. It is similar in the relation \eqref{eq-w2-thry}.}
Then we need to show that the solution $(u,\psi)$ has higher regularity as it is done in the proof of Theorem \ref{thm-wellposedness}. In fact, we have for each $i=1,\dots,d$,
$$
\left\|D_i\mc{G} \ast\rho_s D_iu_s \right\|_2
\leq \left\|D_i\mc{G} \ast\rho_s  \right\|_{\infty}   \left\|D_iu_s \right\|_2
\leq S_d\|\rho_s\|_4  \left\|D_iu_s \right\|_2,
$$
and thus, $ D_i\mc{G} \ast\rho D_iu \in L^2_{\mc F^W} (0,T_1;L^2)$, which by  $L^2$-theory of backward SPDE indicated further  
\begin{align}
(u,\psi)\in   \left( L^2_{\mathcal F^W} (0,T;W^{2,2}(\RR^d)) \cap S_{\mc F^W}^2 ([0,T];W^{1,2}(\RR^d)) \right) \times L^2_{\mathcal F^W} (0,T;W^{1,2}(\RR^d)).
\label{eq-w2-thry}
\end{align}
Taking derivatives gives further
\begin{align*}
\left\|D_j(D_i\mc{G} \ast\rho_s D_iu_s) \right\|_2
&\leq \left\|D_{ij}\mc{G} \ast\rho_s D_iu_s    \right\|_2 + \left\|D_{i}\mc{G} \ast\rho_s D_{ij}u_s    \right\|_2
\\
& \leq
	 \left\|D_{ij}\mc{G} \ast\rho_s\right\|_4\left\|D_iu_s    \right\|_4 	+   \left\|D_{i}\mc{G} \ast\rho_s\right\|_{\infty}  \left\| D_{ij}u_s    \right\|_2
	 \\
&\leq
	 S_d\left\|\rho_s\right\|_4  \left\|D_iu_s    \right\|_2^{1/4}   \left\|D_iu_s    \right\|_6^{3/4} 
	 		+  S_d\left\|  \rho_s\right\|_{4}  \left\| D_{ij}u_s    \right\|_2
	\\
&\leq
	 \left\|\rho_s\right\|_4  \left\| u_s    \right\|_{W^{2,2}}  
	 		+  S_d\left\|  \rho_s\right\|_{4}  \left\| u_s    \right\|_{W^{2,2}},
\end{align*}
and thus, $D_i\mc{G} \ast\rho_s D_iu_s \in L^2_{\mc F^W} (0,T_1;W^{1,2}(\RR^d))$, $i=1,\dots,d$. Applying the $L^2$-theory again, we arrive at
$$(u,\psi)  \in \left( L^2_{\mathcal F^W} (0,T;W^{3,2}(\RR^d)) \cap S_{\mc F^W}^2 ([0,T];W^{2,2}(\RR^d)) \right) \times L^2_{\mathcal F^W} (0,T;W^{2,2}(\RR^d)) .  $$

W.l.o.g., we prove the probabilistic representation \eqref{eq-prob-rep} for the case when $t=0$. In fact, a straightforward application of \cite[Theorem 3.1]{Tang-Yang-2011} yields that
\begin{align*}
u_0(y)=G(\overline Y_{T_1}) -\int_0^{T_1} \left( \sum_k^{d'} \psi^k_s(\overline Y_s) + \sum_{i=1}^d\sigma^{ik}_s(\overline Y_s) D_iu_s(\overline Y_s)\right) \,dW^k_s, \quad \text{a.s.}
\end{align*}
Noticing that by embedding theorem it holds that $L^2_{\mathcal F^W} (0,T;W^{2,2}(\RR^d))   \,\, \hookrightarrow\,\, L^2_{\mathcal F^W} (0,T;C^{1/4}(\RR^d)) $, we may easily check that the stochastic integral in the above equality is mean-zero. Therefore, we have $u_0(y)=\EE\left[  G(\overline Y_{T_1}) \big|\overline Y_0=y,\, \mc F_0^W\right]$ by taking conditional expectation on both sides. For general $t\in(0,T_1]$, the proof of \eqref{eq-prob-rep} follows similarly.

\end{proof}


\section{ Mean-field limit of the particle system \eqref{particle} towards the stochastic KS equation \eqref{SPDE}}

To prove the mean-field limit, we recall the following auxiliary stochastic dynamics $\{(Y_t^i)_{t\geq 0}\}_{i=1}^N$ as defined in \eqref{SDE} 
\begin{align}\label{SDEs}
\begin{cases}
\ud Y_t^i= \chi\nabla\mc{G}\ast\rho_t(Y_t^i)\dt+\nu_t(Y_t^i) \ud B_t^i+\sigma_t(Y_t^i)\ud W_t,\quad t>0,\quad i=1,\cdots,N\,,\\
\rho_t\mbox{ is the conditional density of  }Y_t^i \mbox{ given } \mc{F}_t^W\mbox{ for all } i=1,\cdots,N\,.
\\
Y_0^i=\zeta^i.
\end{cases}
\end{align}
This means that $\{(Y_t^i)_{t\geq 0}\}_{i=1}^N$ are $N$ copies of solutions to the nonlinear SDE \eqref{SDE1}, and they are conditional i.i.d. given $W_t$.  
 We will also use the regularized version
\begin{align}\label{RSDEs}
	\begin{cases}
		\ud Y_t^{i,\varepsilon}= \chi\nabla\mc{G}_\varepsilon\ast\rho_t^\varepsilon(Y_t^{i,\varepsilon})\dt+\nu_t(Y_t^{i,\varepsilon}) \ud B_t^i+\sigma_t(Y_t^{i,\varepsilon})\ud W_t,\quad t>0,\quad i=1,\cdots,N\,,\\
		\rho_t^\varepsilon\mbox{ is the conditional density of  }Y_t^{i,\varepsilon} \mbox{ given } \mc{F}_t^W\mbox{ for all } i=1,\cdots,N\,.
		\\
		Y_0^{i,\varepsilon}=Y_0^i=\zeta^i,
	\end{cases}
\end{align}
with $\rho_t^\varepsilon$ satisfying the following regularized stochastic KS equation
\begin{align}\label{RSPDE}
\begin{cases}
\ud \rho_t^\varepsilon
=&\frac{1}{2}\sum_{i,j=1}^dD_{ij}\left(\rho_t^\varepsilon\sum_{k=1}^{d'}(\nu^{ik}_t\nu^{jk}_t+ \sigma^{ik}_t\sigma^{jk}_t)\right)\dt-\chi\nabla\cdot(\nabla (\mathcal{G}_\varepsilon*\rho_t^\varepsilon)\rho_t^\varepsilon)\dt\\
&-\sum_{i=1}^{d}D_i\left( \rho_t^\varepsilon \sum_{k=1}^{d'}\sigma^{ik}_t\ud W_t^k\right)\,, \\
\rho_0^\varepsilon=\rho_0\,,
\end{cases}
\end{align}
Indeed, following the same arguments as in Sections 3-4,  we obtain the well-posedness of the regularized system \eqref{RSDEs} and equation \eqref{RSPDE}. Next we estimate the difference of the solutions. Set $e_t^\varepsilon=\rho_t^\varepsilon-\rho_t$ for $t\in[0,T]$  with $e_0^\varepsilon=0$. Following the same computation as in \eqref{eq:ito-delta}, one has
\begin{align}\label{54}
\|e_t^\varepsilon\|_{4}^4 
&\leq
-6\lambda  \int_0^{t} \|e_s^\varepsilon\nabla e_s^\varepsilon\|_2^2\,ds 
+2 \lambda\int_0^t\|e_s^\varepsilon\nabla e_s^\varepsilon\|_{2}^2\,ds+\frac{(12\Lambda^2 )^2}{2\lambda}\int_0^t\| e_s^\varepsilon \|_{4}^4\,ds
\nonumber\\
&\quad
+\int_0^t12\chi\left\langle |e_s^\varepsilon|^2  \nabla e_s^\varepsilon, (\nabla\mc{G}_\varepsilon\ast\rho_s^\varepsilon)\rho_s^\varepsilon- (\nabla\mc{G}\ast \rho_s) \rho_s  \right\rangle\,ds,\quad\text{a.s.}\,.
\end{align}
Notice that 
\begin{align*}
&12\chi\left\langle |e_s^\varepsilon|^2  \nabla e_s^\varepsilon, (\nabla\mc{G}_\varepsilon\ast\rho_s^\varepsilon)\rho_s^\varepsilon- (\nabla\mc{G}\ast \rho_s) \rho_s  \right\rangle\\
=&12\chi\left\langle |e_s^\varepsilon|^2  \nabla e_s^\varepsilon, \nabla\mc{G}\ast e_s^\varepsilon \rho^{\varepsilon}_s+ \nabla\mc{G}\ast \rho_s e_s^\varepsilon  \right\rangle
+12\chi\left\langle |e_s^\varepsilon|^2  \nabla e_s^\varepsilon, (\nabla\mc{G}_\varepsilon-\nabla\mc{G})\ast \rho_s^\varepsilon\rho_s^\varepsilon \right\rangle\,.
\end{align*}
Similar to the computation in \eqref{314}, one obtains
\begin{align*}
12\chi\left\langle |e_s^\varepsilon|^2  \nabla e_s^\varepsilon, \nabla\mc{G}\ast e_s^\varepsilon \rho^{\varepsilon}_s+ \nabla\mc{G}\ast \rho_s e_s^\varepsilon  \right\rangle
\leq 
2\lambda  \|e_s^\varepsilon\nabla e_s^\varepsilon\|_{2}^2 +C(T,\chi,\lambda,\Lambda,d)\|e_s^\varepsilon\|_{4}^4\,.
\end{align*}
On the other hand, we compute
\begin{align*}
12\chi\left\langle |e_s^\varepsilon|^2  \nabla e_s^\varepsilon, (\nabla\mc{G}_\varepsilon-\nabla\mc{G})\ast \rho_s^\varepsilon\rho_s^\varepsilon \right\rangle
\leq12\chi\|e_s^\varepsilon\nabla e_s^\varepsilon\|_{2}\norm{(\nabla\mc{G}_\varepsilon-\nabla\mc{G})\ast \rho_s^\varepsilon}_\infty\|e_s^\varepsilon\|_4\|\rho_s^\varepsilon\|_4\,.
\end{align*}
Notice that
\begin{align}
|(\nabla\mc{G}_\varepsilon-\nabla\mc{G})\ast \rho_s^\varepsilon |(x)&= |\psi_\varepsilon\ast (\nabla\mc{G}\ast \rho_s^\varepsilon)-\nabla\mc{G}\ast \rho_s^\varepsilon|(x)=\left|\int_{\R^d}\psi_\varepsilon(y)[\nabla\mc{G}\ast \rho_s^\varepsilon(x-y)-\nabla\mc{G}\ast \rho_s^\varepsilon(x)]\dy\right|\notag\\
&\leq \norm{\nabla\mc{G}\ast \rho_s^\varepsilon}_{W^{1,\infty}}\int_{\RR^d}|y|\psi_\varepsilon(y)\dy\leq C\varepsilon \norm{\nabla\mc{G}\ast \rho_s^\varepsilon}_{W^{1,\infty}}\leq C\varepsilon \norm{\rho_s^\varepsilon}_{W^{1,4}}\,,
\end{align}
where $C$ depends only on $T,\chi,\lambda,\Lambda,$ and $ d$.
Then one has
\begin{equation}
12\chi\left\langle |e_s^\varepsilon|^2  \nabla e_s^\varepsilon, (\nabla\mc{G}_\varepsilon-\nabla\mc{G})\ast \rho_s^\varepsilon\rho_s^\varepsilon \right\rangle
\leq
2 \lambda \|e_s^\varepsilon\nabla e_s^\varepsilon\|_{2}^2+C\varepsilon^2\norm{\rho_s^\varepsilon}_{W^{1,4}}^2\|e_s^\varepsilon\|_4^2\|\rho_s^\varepsilon\|_4^2,
\end{equation}
and thus
\begin{equation}
12\chi\left\langle |e_s^\varepsilon|^2  \nabla e_s^\varepsilon, (\nabla\mc{G}_\varepsilon\ast\rho_s^\varepsilon)\rho_s^\varepsilon- (\nabla\mc{G}\ast \rho_s) \rho_s  \right\rangle
\leq 
4\lambda \|e_s^\varepsilon\nabla e_s^\varepsilon\|_{2}^2+C\varepsilon^2\norm{\rho_s^\varepsilon}_{W^{1,4}}^2\|e_s^\varepsilon\|_4^2\|\rho_s^\varepsilon\|_4^2+C\norm{e_s^\varepsilon}_4^4\,.
\end{equation}
It follows from \eqref{54} that
\begin{align*}
\|e_t^\varepsilon\|_{4}^4 \leq  C_1\int_0^t\|e_s^\varepsilon\|_4^4\ds+C_2\varepsilon^2\int_0^t\norm{\rho_s^\varepsilon}_{W^{1,4}}^2\|e_s^\varepsilon\|_4^2\|\rho_s^\varepsilon\|_4^2\ds\,,
\end{align*} 
where $C_1,C_2$ depend only on $T,\chi,\lambda,\Lambda$, and $ d$.
 By Gronwall's inequality, we have
 \begin{equation}
 \sup_{t\in[0,T]}\|e_t^\varepsilon\|_{4}^4\leq \left[C_2\varepsilon^2 \sup_{t\in[0,T]}\{\|e_t^\varepsilon\|_4^2\|\rho_t^\varepsilon\|_4^2\}\int_0^T\norm{\rho_s^\varepsilon}_{W^{1,4}}^2\ds\right] \exp(C_1T)\quad \text{a.s.}\,.
 \end{equation}
 This leads to 
\begin{align}
\|\rho^\varepsilon-\rho\|_{S_{\mc F^W}^{2}([0,T];L^4(\RR^d))}
&\leq  C\left(C_1,C_2,T,\|\rho\|_{S_{\mc F^W}^{\infty}([0,T];L^4(\RR^d))},\|\rho^\varepsilon\|_{S_{\mc F^W}^{\infty}([0,T];L^4(\RR^d))},\|\rho^\varepsilon\|_{L_{\mc F^W}^{2}([0,T];W^{1,4}(\RR^d))}\right)\varepsilon
\nonumber \\
&\leq 
C\left( T, \chi,\lambda,\Lambda,d\right) \cdot \varepsilon
  \label{error}
\end{align}
where we have used the fact that the quantities $\|\rho\|_{S_{\mc F^W}^{\infty}([0,T];L^4(\RR^d))},\|\rho^\varepsilon\|_{S_{\mc F^W}^{\infty}([0,T];L^4(\RR^d))}$, and $\|\rho^\varepsilon\|_{L_{\mc F^W}^{2}([0,T];W^{1,4}(\RR^d))}$ depend only on $T,\chi,\lambda,\Lambda$, and $d$, independent of $\varepsilon$.

Our main theorem of mean-field limit states that the mean-field dynamics $\{(Y_t^{i,\varepsilon})_{t\geq 0}\}_{i=1}^N$ well approximate the regularized interacting particle system $\{(X_t^{i,\varepsilon})_{t\geq 0}\}_{i=1}^N$ in \eqref{particle}.
\begin{thm}\label{thmmean}
Under the same assumptions as in Theorem \ref{thm-wellposedness}, let  $\{(X_t^{i,\varepsilon})_{t\geq 0}\}_{i=1}^N$ and $\{(Y_t^{i,\varepsilon})_{t\geq 0}\}_{i=1}^N$  satisfy the interacting particle system \eqref{particle} and the mean-field dynamics \eqref{RSDEs} respectively. 
Then for any fixed $0<\delta\ll1$, such that $\varepsilon^{-d}\leq \delta \ln (N)$ and $C\delta<1$ it holds that 
\begin{align}\label{eqmean}
 \sup\limits_{t\in[0,T]}\sup\limits_{i=1,\cdots,N}\EE\left[\left|X_t^{i,\varepsilon}-Y_t^{i,\varepsilon}\right|^2\right] \leq C\frac{(\delta\ln(N))^{\frac{2d-2}{d}} }{N^{1-C\delta}}\,,
\end{align}
where $C$ is a constant depending only on $\chi,T,d,d'$ and $\Lambda$.
\end{thm}
\begin{proof}
Applying It\^{o}'s formula yields that
\begin{align}
 |X_t^{i,\varepsilon}-Y_t^{i,\varepsilon}|^2&= \int_0^t 2\chi(X_s^{i,\varepsilon}-Y_s^{i,\varepsilon})\cdot \left(\frac{1}{N-1}\sum\limits_{j\neq i}^N \nabla\mc{G}_\varepsilon(X_s^{i,\varepsilon}-X_s^{j,\varepsilon})- \nabla\mc{G}_\varepsilon\ast\rho_s^{\varepsilon}(Y_s^{i,\varepsilon})\right)\ds\notag\\
&\quad+\int_0^t2(X_s^{i,\varepsilon}-Y_s^{i,\varepsilon})\cdot (\nu_s(X_s^{i,\varepsilon})-\nu_s(Y_s^{i,\varepsilon})) \ud B_s^i+\int_0^t2(X_s^{i,\varepsilon}-Y_s^{i,\varepsilon})\cdot (\sigma_s(X_s^{i,\varepsilon})-\sigma_s(Y_s^{i,\varepsilon}))\ud W_s \notag\\
&\quad+\int_0^t\sum_j^d\sum_{k=1}^{d'}\left(\nu^{jk}_s(X_s^{i,\varepsilon})-\nu^{jk}_s(Y_s^{i,\varepsilon})\right)^2\ds+\int_0^t\sum_j^d\sum_{k=1}^{d'}\left(\sigma^{jk}_s(X_s^{i,\varepsilon})-\sigma^{jk}_s(Y_s^{i,\varepsilon})\right)^2\ds\,.\notag
\end{align}
Taking expectations on both sides one has
\begin{align}\label{estotal}
\EE\left[|X_t^{i,\varepsilon}-Y_t^{i,\varepsilon}|^2\right]&\leq\EE\left[\int_0^t 2\chi(X_s^{i,\varepsilon}-Y_s^{i,\varepsilon})\cdot \left(\frac{1}{N-1}\sum\limits_{j\neq i}^N \nabla\mc{G}_\varepsilon(X_s^{i,\varepsilon}-X_s^{j,\varepsilon})- \nabla\mc{G}_\varepsilon\ast\rho_s^{\varepsilon}(Y_s^{i,\varepsilon})\right)\ds\right] \notag\\
&\quad+C(d,d',\Lambda)\int_0^t\EE\left[|X_s^{i,\varepsilon}-Y_s^{i,\varepsilon}|^2\right]\ds\,,
\end{align}
where we have used the fact that
\begin{equation*}
\EE\left[\int_0^t2(X_s^{i,\varepsilon}-Y_s^{i,\varepsilon})\cdot (\nu_s(X_s^{i,\varepsilon})-\nu_s(Y_s^{i,\varepsilon})) \ud B_s^i\right]=\EE\left[\int_0^t2(X_s^{i,\varepsilon}-Y_s^{i,\varepsilon})\cdot (\sigma_s(X_s^{i,\varepsilon})-\sigma_s(Y_s^{i,\varepsilon}))\ud W_s\right]=0\,,
\end{equation*}
and $(ii)$ in Assumption \ref{assum1}.

To continue, we split the error
\begin{equation*}
\EE\left[\int_0^t 2\chi(X_s^{i,\varepsilon}-Y_s^{i,\varepsilon})\cdot \left(\frac{1}{N-1}\sum\limits_{j\neq i}^N \nabla\mc{G}_\varepsilon(X_s^{i,\varepsilon}-X_s^{j,\varepsilon})- \nabla\mc{G}_\varepsilon\ast\rho_s^{\varepsilon}(Y_s^{i,\varepsilon})\right)\ds\right]
\end{equation*}
into three parts. Notice that 
	\begin{align}
	&\frac{1}{N-1}\sum\limits_{j\neq i}^N \nabla\mc{G}_\varepsilon(X_s^{i,\varepsilon}-X_s^{j,\varepsilon})- \nabla\mc{G}_\varepsilon\ast\rho_s^{\varepsilon}(Y_s^{i,\varepsilon})\notag\\
	&= \frac{1}{N-1} \left(\sum\limits_{j\neq i}^N \nabla\mc{G}_\varepsilon(X_s^{i,\varepsilon}-X_s^{j,\varepsilon})- \sum\limits_{j\neq i}^N \nabla\mc{G}_\varepsilon(Y_s^{i,\varepsilon}-Y_s^{j,\varepsilon})\right)\notag\\
	&\quad
		+\frac{1}{N-1} \sum\limits_{j\neq i}^N \nabla\mc{G}_\varepsilon(Y_s^{i,\varepsilon}-Y_s^{j,\varepsilon})- \nabla\mc{G}_\varepsilon\ast\rho_s^{\varepsilon}(Y_s^{i,\varepsilon})\notag\\
	&=:I_{11}^s+I_{12}^s\,. \notag
	\end{align}
	First we compute
	\begin{align*}
	\int_0^t 2\chi(X_s^{i,\varepsilon}-Y_s^{i,\varepsilon})\cdot I_{11}^s \ds&\leq 	2\chi\int_0^t |X_s^{i,\varepsilon}-Y_s^{i,\varepsilon}| \frac{1}{N-1}\norm{\nabla\mc{G}_\varepsilon}_{W^{1,\infty}}\sum\limits_{j=1}^N\left|X_s^{j,\varepsilon}-Y_s^{j,\varepsilon} \right|\ds\\
	&\leq\frac{C\varepsilon^{-d}}{N-1}\int_0^t\sum\limits_{j=1}^N\left|X_s^{j,\varepsilon}-Y_s^{j,\varepsilon} \right|^2\ds\,,
	\end{align*}
which leads to
	\begin{align}
	\EE\left[\int_0^t 2\chi(X_s^{i,\varepsilon}-Y_s^{i,\varepsilon})\cdot I_{11}^s \ds\right]
	& \leq \frac{C\varepsilon^{-d}}{N-1}\int_0^t\sum\limits_{j=1}^N\EE\left[\left|X_s^{j,\varepsilon}-Y_s^{j,\varepsilon} \right|^2\right]\ds
	\notag\\
		& \leq C\varepsilon^{-d}\int_0^t\sup\limits_{i=1,\cdots,N}\EE\left[\left|X_s^{i,\varepsilon}-Y_s^{i,\varepsilon}\right|^2\right]\ds\,,\label{esI11}
	\end{align}
	where $C$ depends only on $\chi$ and $d$.

	To estimate the second term, we rewrite
	\begin{align*}
	I_{12}^s= \frac{1}{N-1}\sum\limits_{j\neq i}^N\left(\nabla\mc{G}_\varepsilon(Y_s^{i,\varepsilon}-Y_s^{j,\varepsilon})- \nabla\mc{G}_\varepsilon\ast\rho_s^{\varepsilon}(Y_s^{i,\varepsilon})\right)=:\frac{1}{N-1}\sum\limits_{j\neq i}^NZ_j^i\,,
	\end{align*}
	where 
	\begin{equation*}
	Z_j^i=\nabla\mc{G}_\varepsilon(Y_s^{i,\varepsilon}-Y_s^{j,\varepsilon})- \nabla\mc{G}_\varepsilon\ast\rho_s^{\varepsilon}(Y_s^{i,\varepsilon}),\quad j\neq i\,.
	\end{equation*}
	It is easy to check that
	\begin{equation*}
	\EE[Z_j^i|\mc{F}_t^W,Y_s^{i,\varepsilon}]=\nabla\mc{G}_\varepsilon\ast\rho_s^{\varepsilon}(Y_s^{i,\varepsilon})- \nabla\mc{G}_\varepsilon\ast\rho_s^{\varepsilon}(Y_s^{i,\varepsilon})=0\,,
	\end{equation*}
	since $\{Y_s^{j,\varepsilon}\}_{j=1}^N$ are conditional i.i.d. with common conditional density $\rho_s^{\varepsilon}$ given $\mc{F}_t^W$. Thus one concludes that
	\begin{align}
	\EE[|I_{12}^s|^2]&=\frac{1}{(N-1)^2}\EE\left[\left(\sum\limits_{j\neq i}^NZ_j^i\right)\left(\sum\limits_{k\neq i}^NZ_k^i\right)\right]\nn\\
	&=\frac{1}{(N-1)^2}\EE\left[\EE\left[\left(\sum\limits_{j\neq i}^NZ_j^i\right)\left(\sum\limits_{k\neq i}^NZ_k^i\right)|\mc{F}_t^W,Y_s^{i,\varepsilon}\right]\right] \nn\\
	&=\frac{1}{(N-1)^2}\EE\left[\EE\left[\sum\limits_{j\neq i}^N|Z_j^i|^2|\mc{F}_t^W,Y_s^{i,\varepsilon}\right]\right] 
	=\frac{1}{N-1}\EE[|Z_{1}^2|^2]\,. \nn
	\end{align}
	Due to the fact that, using \eqref{espotential},
	\begin{equation*}
\EE[|Z_{2}^1|^2]=\EE\left[(\nabla\mc{G}_\varepsilon(Y_s^{1,\varepsilon}-Y_s^{2,\varepsilon})- \nabla\mc{G}_\varepsilon\ast\rho_s^{\varepsilon}(Y_s^{1,\varepsilon}))^2\right]\leq 4\|\nabla\mc{G}_\varepsilon\|_\infty^2\leq C\varepsilon^{-2(d-1)},
	\end{equation*}
	one has
	\begin{equation*}
	\EE[|I_{12}^s|^2]\leq \frac{C\varepsilon^{-2(d-1)}}{N-1}\,.
	\end{equation*}
	Thus we concludes
	\begin{align}\label{esI12}
	&\EE\left[\int_0^t 2\chi(X_s^{i,\varepsilon}-Y_s^{i,\varepsilon})\cdot I_{12}^s \ds\right]\notag\\
	\leq &\int_0^t \EE\left[|X_s^{i,\varepsilon}-Y_s^{i,\varepsilon}|^2\right] \ds+\int_0^t\chi^2 \EE\left[|I_{12}^s|^2\right] \ds\leq \int_0^t\EE\left[|X_s^{i,\varepsilon}-Y_s^{i,\varepsilon}|^2\right] \ds+\frac{C\varepsilon^{-2(d-1)}}{N-1}\,, \nn
	\end{align}
where $C$ depends only on $\chi$ and $d$.

Now collecting estimates \eqref{esI11} and \eqref{esI12} implies
\begin{align}
&\EE\left[\int_0^t 2\chi(X_s^{i,\varepsilon}-Y_s^{i,\varepsilon})\cdot \left(\frac{1}{N-1}\sum\limits_{j\neq i}^N \nabla\mc{G}_\varepsilon(X_s^{i,\varepsilon}-X_s^{j,\varepsilon})- \nabla\mc{G}_\varepsilon\ast\rho_s^{\varepsilon}(Y_s^{i,\varepsilon})\right)\ds\right]\notag\\
\leq& C\varepsilon^{-d}\int_0^t\sup\limits_{i=1,\cdots,N}\EE\left[\left|X_s^{i,\varepsilon}-Y_s^{i,\varepsilon}\right|^2\right]\ds+\frac{C\varepsilon^{-2(d-1)}}{N-1}
\end{align}
which together with \eqref{estotal} lead to
\begin{equation*}
\sup\limits_{i=1,\cdots,N}\EE\left[\left|X_t^{i,\varepsilon}-Y_t^{i,\varepsilon}\right|^2\right]\leq C_1\varepsilon^{-d}\int_0^t\sup\limits_{i=1,\cdots,N}\EE\left[\left|X_s^{i,\varepsilon}-Y_s^{i,\varepsilon}\right|^2\right]\ds+\frac{C_2\varepsilon^{-2(d-1)}}{N-1}\,.
\end{equation*}
Applying Gronwall's inequality further yields that
\begin{align}
\sup\limits_{t\in[0,T]}\sup\limits_{i=1,\cdots,N}\EE\left[\left|X_t^{i,\varepsilon}-Y_t^i\right|^2\right]
\leq\frac{C_2\varepsilon^{-2(d-1)}}{N-1}e^{C_1\varepsilon^{-d}T}\leq C\frac{(\delta\ln(N))^{\frac{2d-2}{d}} }{N^{1-C\delta}}\,, \notag
\end{align}
where we let $e^{\varepsilon^{-d}}\leq N^{\delta}$, i.e. $\varepsilon^{-d}\leq \delta \ln (N)$,  for any fixed $0<\delta<\frac{1}{C}$. The proof is completed.
\end{proof}

{Theorem \ref{thmmean} implies the convergence in law of the empirical measure in the following sense:
\begin{cor}\label{cormean}
	Under the same assumptions as in Theorem \ref{thmmean}, the empirical measure 
		\begin{equation}
	\rho_t^{\varepsilon,N}:=\frac{1}{N}\sum_{i=1}^{N}\delta_{X_t^{i,\varepsilon}}
	\end{equation}
	associated to the stochastic particle system \eqref{particle} converges weakly to  unique   solution $\rho_t$ to the nonlinear SPDE \eqref{SPDE}.
	More precisely,  for any fixed $0<\delta\ll1$, such that $\varepsilon^{-d}= \delta \ln (N)$ and $C\delta<1$, it holds that  for all $t\in[0,T]$
	\begin{equation}
\EE\left[|\la \rho_t^{\varepsilon,N},\phi\ra-\la \rho_t,\phi\ra|^2\right]\leq C\left(\frac{(\delta\ln(N))^{\frac{2d-2}{d}} }{N^{1-C\delta}}+\frac{1}{N}+(\delta \ln (N))^{-\frac{2}{d}}\right)\,,
	\end{equation}
for any $\phi\in C_c^1(\RR^d)$, where $C$  depends only on $\norm{\phi}_{C^1}$, $\chi,T,\lambda,\Lambda,d$, and $\norm{\rho_0}_{ W^{2,2}(\RR^d)}$.
\end{cor}
\begin{proof}
Let us compute
\begin{align}
\EE\left[|\la \rho_t^{\varepsilon,N},\phi\ra-\la \rho_t^\varepsilon,\phi\ra|^2\right]&=\EE\left[\left|\frac{1}{N}\sum_{i=1}^{N}\phi(X_t^{i,\varepsilon})-\int_{\RR^d}\phi(x)\rho_t^\varepsilon(x)\dx\right|^2\right]\nn\\
&\leq 2\EE\left[|\phi(X_t^{1,\varepsilon})-\phi(Y_t^{1,\varepsilon})|^2\right]+2\EE\left[\left|\frac{1}{N}\sum_{i=1}^{N}\phi(Y_t^{i,\varepsilon})-\int_{\RR^d}\phi(x)\rho_t^\varepsilon(x)\dx\right|^2\right] \nn\\
&=: I_1+I_2\,.
\end{align}
According to \eqref{eqmean}, one has
\begin{equation}\label{esI1}
I_1\leq 2\norm{\nabla\phi}_{\infty}^2\EE\left[|X_t^{1,\varepsilon}-Y_t^{1,\varepsilon}|^2\right]\leq C\frac{(\delta\ln(N))^{\frac{2d-2}{d}} }{N^{1-C\delta}}\,,
\end{equation}
where $C$  depends only on $\norm{\nabla\phi}_{\infty}$, $\chi,T,\lambda,\Lambda,d$ and $\norm{\rho_0}_{W^{2,2}(\RR^d)}$. To estimate $I_2$, we compute that
\begin{align}
&\EE\left[\left|\frac{1}{N}\sum_{i=1}^{N}\phi(Y_t^{i,\varepsilon})-\int_{\RR^d}\phi(x)\rho_t^\varepsilon(x)\dx\right|^2\right]
\leq\, \frac{1}{N^2}\sum_{i=1}^{N}\EE\left[\left|\phi(Y_t^{i,\varepsilon})-\int_{\RR^d}\phi(x)\rho_t^\varepsilon(x)\right|^2\right]
\leq\, C\frac{1}{N}\,,
\end{align}
where $C$ depends only on $\norm{\phi}_{\infty}$. This combined with \eqref{esI1} implies 
	\begin{equation}\label{coreq}
\EE\left[|\la \rho_t^{\varepsilon,N},\phi\ra-\la \rho_t^\varepsilon,\phi\ra|^2\right]\leq C\left(\frac{(\delta\ln(N))^{\frac{2d-2}{d}} }{N^{1-C\delta}}+\frac{1}{N} \right)\,,
\end{equation}

 Next, using \eqref{error} we compute
\begin{align}
\EE\left[|\la \rho_t^{\varepsilon},\phi\ra-\la \rho_t,\phi\ra|^2\right]\leq C\EE\left[\sup_{t\in[0,T]}\norm{\rho^\varepsilon_t-\rho_t}_4^2\right]
=C\norm{\rho^\varepsilon-\rho}_{S_{\mc F^W}^{2}([0,T];L^4(\RR^d))}^2\leq C\varepsilon^2\,.
\end{align}
Hence one has
\begin{align}
&\EE\left[|\la \rho_t^{\varepsilon,N},\phi\ra-\la \rho_t,\phi\ra|^2\right]\leq 2\EE\left[|\la \rho_t^{\varepsilon,N},\phi\ra-\la \rho_t^\varepsilon,\phi\ra|^2\right]+2\EE\left[|\la \rho_t^{\varepsilon},\phi\ra-\la \rho_t,\phi\ra|^2\right]\notag\\
\leq& C\left(\frac{(\delta\ln(N))^{\frac{2d-2}{d}} }{N^{1-C\delta}}+\frac{1}{N}+(\delta \ln (N))^{-\frac{2}{d}} \right)\,.
\end{align}
This completes the proof.
\end{proof}}

\bibliographystyle{abbrv}
\bibliography{stochastic}


\end{document}